\documentclass[11pt,a4paper]{article}

\usepackage[a4paper, total={5.8in, 9in}]{geometry}

\usepackage{amsmath,amsfonts}
\usepackage{amsthm,amssymb,hyperref}
\usepackage[english]{babel}
\usepackage{enumitem}
\usepackage{mdframed}

\usepackage{xcolor} 
\hypersetup{
    colorlinks,
    linkcolor={blue},
    citecolor={red},
    urlcolor={blue}
}

\newenvironment{cproof}
{\begin{proof}
 [Proof.]
 \vspace{-1.5\parsep}
}
{ \end{proof}}

\newtheorem{theorem}{Theorem}
\newtheorem{lemma}[theorem]{Lemma}
\newtheorem{conjecture}[theorem]{Conjecture}
\newtheorem{corollary}[theorem]{Corollary}
\newtheorem{proposition}[theorem]{Proposition}

\newtheorem{claim}[theorem]{Claim}
\newtheorem{definition}[theorem]{Definition}
\numberwithin{equation}{section}
\numberwithin{theorem}{section}

\newcommand{\hh}{\hat{H}}
\newcommand{\hk}{\hat{T}_k}
\newcommand{\hj}{\hat{J}}

\begin{document}

\title{Coloring dense digraphs}

\author{Ararat Harutyunyan$^a$, Tien-Nam Le$^b$,\\ Alantha Newman$^c$, and St\'{e}phan Thomass\'{e}$^b$\\~\\
\small $^a$Institut de Math\'ematiques de Toulouse \\ \small Universit\'e Toulouse III \\ \small 31062 Toulouse Cedex 09, France\\~\\
\small $^b$Laboratoire d'Informatique du Parall\'elisme \\ \small   UMR 5668 ENS Lyon - CNRS - UCBL - INRIA
\\ \small Universit\'e de Lyon, France\\~\\
\small $^c$Laboratoire G-SCOP \\ \small CNRS, Universit\'e Grenoble-Alpes \\ \small Grenoble, France}

\date{}

\maketitle

\begin{abstract}
The \emph{chromatic number} of a digraph $D$ is the minimum number of
acyclic subgraphs covering the vertex set of $D$.  A tournament $H$ is
a \emph{hero} if every $H$-free tournament $T$ has chromatic number
bounded by a function of $H$.  Inspired by the celebrated
Erd\H{o}s--Hajnal conjecture, Berger et al. fully characterized the
class of heroes in 2013.  We extend this framework to dense
digraphs: A digraph $H$ is a \emph{superhero} if every $H$-free
digraph $D$ has chromatic number bounded by a function of $H$ and
$\alpha(D)$, the independence number of the underlying graph of
$D$. We prove here that a digraph is a superhero if and only if it is
a hero, and hence characterize all superheroes.  This answers a
question of Aboulker, Charbit and Naserasr.
\end{abstract}


\section{Introduction} \label{sec:intro}

Every digraph in this paper is simple, loopless and finite, where a
digraph $D$ is \emph{simple} if for every two vertices $u$ and $v$ of
$D$, there is at most one arc with endpoints $\{u,v\}$.  Given a
digraph $D$, we denote by $V(D)$ the vertex set of $D$. The
\emph{independence number} $\alpha(D)$ of a digraph $D$ is the
independence number of the underlying graph of $D$. A subset $X$
of $V(D)$ is \emph{acyclic} if the subgraph (i.e., subdigraph) of $D$ induced by $X$
contains no directed cycle.  A \emph{$k$-coloring} of a digraph $D$ is a
partition of $V(D)$ into $k$ acyclic sets, and the {\it chromatic number}
$ \chi(D)$ is the minimum number $k$ for which $D$ admits a $k$-coloring. 
This digraph invariant was introduced by Neumann-Lara \cite{Neu82}, and 
naturally generalizes many results on the graph chromatic number 
(see, for example, \cite{BFJKM04}, \cite{HM11}
\cite{HM11b}, \cite{HM12}, \cite{KLMR13}).

Given digraphs $D$ and $H$, we say that $D$ is \emph{$H$-free} if
there is no induced subgraph of $D$ isomorphic to $H$.  A digraph $T$
is a \emph{tournament} if there is an arc between every pair of
distinct vertices of $T$.  A tournament $H$ is a \emph{hero} if there
is a number $f(H)$ such that $ \chi(T)\le f(H)$ for every $H$-free
tournament $T$. The class of heroes was fully characterized in 2013 by
Berger et al \cite{BCC13}.

One may try to study the more general question for arbitrary digraphs:
for which digraph $H$ does there exist a constant $\varepsilon(H) > 0$
such that every $H$-free digraph $D$ satisfies $\chi(D) \leq
|V(D)|^{1-\varepsilon(H)}$?  In fact, it was conjectured in \cite{HM}
that this holds for every $H$.

\begin{conjecture}\label{conj:poly-subgraph}
For every digraph $H$, there is $\varepsilon>0$ such that if $D$ is a
$H$-free digraph, then $\chi(D) \leq |V(D)|^{1-\varepsilon}$.
\end{conjecture} 

The conjecture is open even when $H$ is the oriented triangle, $C_3$.
In fact, it can be viewed as a strengthening of the Erd\H{o}s--Hajnal
conjecture.  For a comprehensive survey on the Erd\H{o}s--Hajnal
conjecture, we refer the reader to \cite{C14}.  Given a digraph $D$,
we denote by $\beta(D)$ the maximum number of vertices of an acyclic
subset of $V(D)$. The following is one formulation of the
Erd\"os--Hajnal conjecture.

\begin{conjecture}\label{conj:EH}
For every tournament $H$, there exists a constant $\varepsilon(H)>0$
such that every $H$-free tournament $T$ satisfies $\beta(T) \ge
|V(T)|^{\varepsilon(H)}$.
\end{conjecture}

Motivated by Conjecture \ref{conj:poly-subgraph}, we say that a
digraph $H$ is a \emph{superhero} if for every integer $\alpha \ge 1$,
there is a number $g(H,\alpha)$ such that $ \chi(D)\le g(H,\alpha)$
for every $H$-free digraph $D$ with $\alpha(D)\le \alpha$.  There are
several remarks. First, if a tournament $H$ is a superhero, then $H$
is a hero by definition.  Second, if $H$ is not a tournament, then $H$
is not a superhero. Indeed, one can easily construct a tournament $T$
with arbitrarily high chromatic number (see Lemma~\ref{lem:mountain}
for example). Note that $H$ contains two non-adjacent vertices.  Hence
$T$ contains no induced isomorphic copy of $H$, and so $H$ is not a
superhero.  From two remarks above, it follows that every superhero is
a hero.  Aboulker, Charbit and Naserasr \cite{A} conjectured that the
converse is also true.  The main result of this paper is the
affirmation of this fact.

\begin{theorem}\label{thm:superhero-main}
Every hero is a superhero.
\end{theorem}

Given tournaments $H_1,H_2,H_3$, we denote by $H_1\Rightarrow H_2$ the vertex-disjoint union of $H_1$ and $H_2$ with complete arcs from $H_1$ to $H_2$, and $\Delta(H_1,H_2,H_3)$ the vertex-disjoint union of $H_1,H_2,H_3$ with complete arcs from $H_1$ to $H_2$, from $H_2$ to $H_3$, and from $H_3$ to $H_1$.
For every integer $k\ge 1$, let $T_k$ denote the \emph{transitive} tournament on $k$ vertices, i.e. the acyclic tournament on $k$ vertices.
The class of heroes are constructed in \cite{BCC13} as follows:
\begin{itemize}
\item The singleton $T_1$ is a hero.
\item If $H_1$ and $H_2$ are heroes, then $H_1\Rightarrow H_2$ is a hero.
\item If $H$ is a hero, then $\Delta(H,T_k,T_1)$ and $\Delta(H,T_1,T_k)$ are heroes for every $k\ge 1$.
\end{itemize}

The main result of \cite{BCC13} is that any tournament that cannot be
constructed by this process is not a hero. Obviously $T_1$ is a
superhero.  Thus to prove Theorem~\ref{thm:superhero-main}, it
suffices to prove the following two theorems.

\begin{theorem}\label{lem:HH-main}
If $H_1$ and $H_2$ are superheroes, then  $H_1\Rightarrow H_2$ is a superhero.
\end{theorem}

\begin{theorem}\label{lem:h1k-main}
If $H$ is a superhero, then $\Delta(H,T_1,T_k)$ and $\Delta(H,T_k,T_1)$ are superheroes for any $k\ge 1$.
\end{theorem}

One may inquire how large $g(H, \alpha)$ needs to be for particular
digraphs $H$.  For digraphs not containing an oriented triangle, we
believe that the following statement may be true.

\begin{conjecture}\label{conj:poly-C_3}
There is an integer $\ell$  such that if $D$ is a $C_3$-free digraph
with $\alpha(D)\le \alpha$, then $ \chi(D)\le \alpha^\ell$.
\end{conjecture} 
Indeed, if Conjecture \ref{conj:poly-C_3} is true, then every
$C_3$-free digraph $D$ either has an independent (hence, acyclic) set
of size $n^{1/2\ell}$ or has chromatic number at most
$(n^{1/2\ell})^\ell=\sqrt{n}$, and hence has an acyclic set of size
$\sqrt{n}$. Consequently, the special case of Conjecture
\ref{conj:EH} when $H=C_3$ would hold for $\varepsilon:=1/2\ell$.

While targeting a polynomial bound for chromatic number of $C_3$-free digraphs, we could not even achieve an exponential bound. However, by an algorithmic approach, we are able to obtain a factorial bound.

\begin{theorem}\label{thm:c3}
If $D$ is a $C_3$-free digraph with $\alpha(D)\le \alpha$, then $
\chi(D) \leq 35^{\alpha-1}\alpha!$ and such a coloring can be found in
polynomial time.
\end{theorem}

On another front, one may be interested in the chromatic number of
digraphs with simple local structure. It was conjectured in
\cite{BCC13} (Conjecture 2.6) that there is a function $g$ such that
if $T$ is a tournament in which the set of out-neighbors of each
vertex has chromatic number at most $k$, then $\chi(T)\le g(k)$.  The
conjectured was verified for $k=2$ in \cite{CKLST} and for all $k$ in
\cite{HLNT17+}. Here, we prove a generalization to digraphs with
bounded independence number.

\begin{theorem}\label{thm:global}
There is a function $g$ such that if $D$ is a digraph with
$\alpha(D)\le \alpha$ and that the set of out-neighbors of each vertex
has chromatic number at most $k$, then $ \chi(D)\le g(k,\alpha)$.
\end{theorem}

The structure of the paper is as follows. 
 We prove Theorem~\ref{thm:global} in
Section \ref{sec:local-global}, which is the main tool to prove 
Theorem~\ref{thm:superhero-main}. Sections \ref{sec:HH} and \ref{sec:H1K} are
devoted to proving Theorems \ref{lem:HH-main} and \ref{lem:h1k-main}, and
hence complete the proof of Theorem~\ref{thm:superhero-main}.
In Section \ref{sec:triangle-free}, we will prove Theorem~\ref{thm:c3}  
to support Conjecture \ref{conj:poly-C_3}.

\subsection{Notation and remarks}

Given a digraph $D$, we say that $u$ \emph{sees} $v$ and $v$ is
\emph{seen by} $u$ if $uv$ is an arc in $D$.  For every $v\in V(D)$,
we denote by $N_D^+(v)$ (respectively, $N_D^-(v)$) the set of
out-neighbors (respectively, in-neighbors) of $v$ in $D$.  Let $N(v) =
N_D^+(v) \cup N_D^-(v)$.  For every $X\subseteq V(D)$, let
$N^+_D(X)=\bigcup_{v\in X}N^+_D(v)$ (respectively, $N^-_D(X)=\bigcup_{v\in
  X}N^-_D(v)$), the set of vertices seen by (respectively, seeing) at
least one vertex of $X$, and let $M^+_D(X)=\bigcap_{v\in X}N^+_D(v)$
(respectively, $M^-_D(X)=\bigcap_{v\in X}N^-_D(v)$), the set of
vertices seen by (respectively, seeing) all vertices of $X$.  Let $N_D(X)$
denote $N_D^+(X) \cup N_D^-(X)$.  We say that two vertices $u,v$ are
\emph{non-adjacent} if there is no arc with endpoints $\{u,v\}$.  For
every $v\in V(D)$, we denote by $N_D^o(v)$ the set non-adjacent
vertices of $v$ in $D$. For a subset $X$ of $V(D)$, we denote by
$N^o_D(X)$ the set of vertices of $V$ non-adjacent to at least one
vertex of $X$.  When it is clear from context (most of the time), we
omit the subscript $D$ in this notation. We will use throughout the
paper the fact that $V(D)\setminus{X} = M^+_D(X)\cup N^-_D(X)\cup
N^o_D(X)$ for any $X\subseteq V(D)$.

Given a digraph $D$ and a set $X\subseteq V(D)$, we denote by $D[X]$
the subgraph of $D$ induced by $X$. When the context is clear, we
often use $X$ to denote $D[X]$, and say \emph{chromatic number of $X$}
to refer to the chromatic number of $D[X]$.  Given a digraph $D$ we
say that a set $X\subseteq V(D)$ is a \emph{dominating set} of $D$ if
every vertex $v\in V(D)\setminus{X}$ is seen by at least one vertex of
$X$. A subset $Y$ of $V(D)$ is \emph{independent} (or \emph{stable})
if any two distinct vertices in $Y$ are non-adjacent.  Given a digraph
$D$ and two disjoint sets $X,Y\subseteq V(D)$, we denote $X\to_D Y$
(or just $X\to Y$) if there is no arc from $Y$ to $X$ in $D$. A key
observation is that if $X\to Y$, then $\chi(X\cup
Y)=\max\big(\chi(X),\chi(Y)\big)$.

A side remark is that some proofs in this paper that proceed by
induction on $\alpha$ use the fact that if $\alpha(D)\le \alpha$, then
$\alpha(N^o(v))\le \alpha-1$ for every $v$, and thus $ \chi(N^o(v))$
is bounded. In these inductive proofs, we often cite known results on
tournaments for the base case $\alpha=1$. However, our proofs are
indeed self-contained since to prove the base case $\alpha=1$, we just
repeat the same arguments and use the fact that in a tournament,
$N^o(v)=\emptyset$ for any vertex $v$. Hence for example, the proof of
Theorem \ref{lem:h1k-main} can serve as an alternative proof for
Theorem 4.1 in \cite{BCC13} (that if $H$ is a hero, then so are
$\Delta(H,T_k,T_1)$ and $\Delta(H,T_1,T_k)$).


\section{From local to global} \label{sec:local-global}

We start with some observations regarding the size of a dominating set
in an acyclic digraph.

\begin{proposition}\label{prop:dom-acyclic}
An acyclic digraph $D$ has a dominating set which is also a stable
set, and hence has size at most $\alpha(D)$.
\end{proposition}
\begin{proof}
We proceed by induction on $|D|$ to show that every acyclic digraph
$D$ has a dominating set $S$ which is stable. The statement clearly
holds for $|D|=1$. For $|D|>1$, since $D$ is acyclic, there is a
vertex $v$ with no in-neighbors. Then $V(D)\setminus\{v\}=N^+(v)\cup
N^o(v)$.  Applying induction to $D[N^o(v)]$, we obtain a stable
dominating set $S'$ of $D[N^o(v)]$. Clearly $S:=S'\cup\{v\}$ is a
dominating set of $D$.  Note that $S'\subseteq N^o(v)$, and so $S$ is
stable.
\end{proof}

Given $t\ge 1$, a digraph $D$ is {\it $t$-local} if for every vertex
$v$ we have $ \chi (N^+(v))\leq t$. Let $\cal C$ be a class of
digraphs closed under taking subdigraphs.  We say that $\cal C$ is
{\it tamed} if for every integer $k$ there exists $K$ and $\ell$ such
that every digraph $T\in \cal C$ with $ \chi (T)\geq K$ contains a set
$A$ of $\ell$ vertices such that $ \chi (A)\geq k$.  Note that a class
of digraphs with bounded chromatic number is indeed tamed.

The following proposition is straightforward.

\begin{proposition}\label{prop}
Let $D$ be a $t$-local digraph. Then for every subset $X$ of $V(D)$, $
\chi \big(N^+(X)\cup X\big)\leq t|X|$.
\end{proposition}

Let us restate Theorem \ref{thm:global}.
\begin{theorem}\label{thm:local-global-main}
For every pair of positive integers $\alpha$ and $t$, there is a
function $f_\alpha(t)$ such that every $t$-local digraph $D$ with
$\alpha(D)\le \alpha$ has chromatic number at most $f_\alpha(t)$.
\end{theorem}

\begin{proof}
We proceed by induction on $\alpha$. The case $\alpha=1$ was proved in
\cite{HLNT17+}.  Suppose that Theorem \ref{thm:local-global-main} is
true for $\alpha-1$ (i.e., $f_{\alpha-1}(t)$ exists for every $t$).

\begin{claim}
For every $t$, the class of $t$-local digraphs with independence
number at most $\alpha$ is tamed.
\end{claim}
\begin{cproof}

We fix some arbitrary $t$ and show the property by induction on $k$.
The claim is trivial for $k=1$.
Assuming now that $(K,\ell)$ exists for $k$, we want to find
$(K',\ell')$ for $k+1$. For this, we set $s:=K+\ell
f_{\alpha-1}(t)+\ell t$ and
$$K':=2 k(\alpha s+1)(t+f_{\alpha-1}(t)+1),$$ and fix $\ell'$ later. 

Let $D$ be a $t$-local digraph with vertex set $V$ such that
$\alpha(D)\le \alpha$ and $ \chi (D)\geq K'$.  Let $B$ be a dominating
set of $D$ of minimum size $b$.  By Proposition \ref{prop}, we have
$$ \chi (D)=\chi (B\cup N^+(B))\leq \chi(B)+\chi(N^+(B))\le (t+1)b.$$
In particular, $b\geq K'/(t+1)\ge 2k(\alpha s+1)$. Consider a subset
$W$ of $B$ of size $k(\alpha s+1)$.  By Proposition \ref{prop}, we
have $ \chi (N^+(W))\leq kt(\alpha s +1)$.  By induction hypothesis on
$\alpha-1$, for every $x\in W$, we have $ \chi(N^o(x))\le
f_{\alpha-1}(t)$ since $N^o(x)$ is $t$-local and $\alpha(N^o(x))\le
\alpha-1$.  Hence, (recalling that $M^-(X):=\bigcap_{x\in X}N^-(x)$
for any $X\subseteq D$)
\begin{align*}
\chi(M^-(W))\ge&\  \chi(D)-\chi(N^+(W))-\chi(N^o(W))-\chi(W)\\
\ge&\  K'- kt(\alpha s+1) -\sum_{x\in W}\chi(N^o(x))-|W| \\
\ge&\  K'- k(\alpha s+1)(t+f_{\alpha-1}(t)+1) \\
\ge &\ K'/2 \\
\ge &\ K.
\end{align*}
In particular, by the tamed property applied to $k$,
one can find a set $A \subset M^-(W)$ such that $A$ has $\ell$ vertices 
and $ \chi (A)\geq k$.

Consider now a subset $S$ of $W$ of size $\alpha s+1$. We 
claim that $ \chi (N^+(S))\geq s$. 
If not, we can cover $N^+(S)$ by at most $s-1$ acyclic sets. 
Since every acyclic set has independence number
at most $\alpha$, it has a dominating set of size at most $\alpha$ by 
Proposition \ref{prop:dom-acyclic}. Hence $N^+(S)$ has a dominating set, 
say $S'$ of size at most $\alpha(s-1)\le |S|-2$.
But this yields a contradiction since the set $(B\setminus S)\cup S'\cup \{x\}$, where $x$
is an arbitrary vertex in $A$, would be a dominating set of $T$ of size less than $|B|$.
Therefore, $ \chi (N^+(S))\geq s$. 

By Proposition \ref{prop} applied 
to $N^+(A)$, we have $ \chi(N^+(A))\le \ell t$. Hence 
\begin{align*}
\chi(N^+(S)\cap M^-(A))\ge&\   \chi (N^+(S))- \chi(N^+(A))-\chi(N^o(A)),\\
\ge &\ s-\ell t-\sum_{x\in A}\chi(N^o(x)),\\
\ge &\ (K+\ell f_{\alpha-1}(t)+\ell t)-\ell t-|A| f_{\alpha-1}(t),\\
= &\ K.
\end{align*}
Thus, by the tamed property applied to $k$, there is a subset $A_S$ of $N^+(S)\cap M^-(A)$ such 
that $|A_S|=\ell$ and $ \chi (A_S)\geq k$.

We now construct our subset of $V$ with chromatic number at least $k+1$. 
For this we consider the set $A'$ consisting of vertices 
$A\cup W$ to which we add the collection of $A_S$, 
for all subsets $S\subseteq W$ of size $\alpha s +1 $. Observe that the number of vertices of $A'$ is at most 
$$\ell':=\ell+k(\alpha s+1)+\ell {k(\alpha s +1) \choose \alpha s +1 }.$$ 

To conclude, it is sufficient to show that 
$ \chi (A') \geq k + 1$. Suppose not, and for contradiction, take a $k$-coloring of $A'$. 
Since
$|W|=k(\alpha s +1)$ there is a monochromatic set $S$ in $W$ of size $\alpha s +1$
(say, colored 1). Recall that
$A_S\subseteq M^-(A)$ and $A\subseteq M^-(W)\subseteq M^-(S)$, so we have all arcs from $A_S$ to $A$ and all arcs from $A$ to $S$, and 
note that since $\chi(A)\ge k$ and $\chi(A_S)\ge k$, both $A$ and $A_S$ have a vertex of each of the $k$ colors. Hence 
there are $u\in A$ and $w\in A_S$ colored 1. Since $A_S\subseteq N^+(S)$,
there is $v\in S$ such that $vw$ is an arc. We then obtain the monochromatic cycle $uvw$ of color 1, a contradiction. 
Thus, $ \chi (A') \geq k+1$, completing the proof of the claim.
\end{cproof}

We now can finish the proof of the theorem.
Since the class of $t$-local digraphs with independence number at most $\alpha$ is tamed, 
by applying tamed property for $k=t+f_{\alpha-1}(t)+1$, we have that there exist $(K,\ell)$ such that every 
$t$-local digraph $D$ with $\alpha(D)\le \alpha$ and  $ \chi(D)\geq K$ contains a set $A$ of 
$\ell$ vertices and $ \chi(A)\geq t+f_{\alpha-1}(t)+1$. We claim that $A$ is a dominating set.
If not, then there is a vertex $v$ such that $A\subseteq N^o(v)\cup N^+(v)$. 
Then $\chi(A)\le \chi( N^o(v))+\chi( N^+(v))\le t+f_{\alpha-1}(t)$, a contradiction. 
Hence, $A$ is a dominating set
of $D$. Thus, $ \chi(D) = \chi (\cup_{x \in A} (N^{+}(x) \cup \{x\})) \leq (t+1)|A|=\ell+\ell t$. 
Consequently, $t$-local digraphs have chromatic number at most $f(t):=\max(K,\ell +\ell t)$.
\end{proof}

By reversing the directions of all arcs, we have the following theorem.
\begin{theorem}\label{thm:local-global-main-in}
For every $t$ and $\alpha$, there is $c$ such that every directed graph $D$ with $\alpha(D)\le \alpha$ and $\chi(N^-(v))\le t$ for every $v\in V(D)$ has chromatic number at most $c$.
\end{theorem}

\section{Proof of Theorem~\ref{lem:HH-main}} \label{sec:HH}
Theorem~\ref{lem:HH-main} states that if $H_1$ and $H_2$ are
superheroes, then so is $H_1\Rightarrow H_2$. We will reuse the
notions of $r$-mountains and $(r,s)$-cliques introduced in
\cite{BCC13}.  Let us first give the idea of the proof of
Theorem~\ref{lem:HH-main} for a special case: $(C_3\Rightarrow
C_3)$-free tournaments have bounded chromatic number.  Given a
$(C_3\Rightarrow C_3)$-free tournament $T$, suppose that there is a a
small set $Q$ in $T$ with chromatic number 3. Then for any partition
of $Q$ into $Q_1,Q_2$, at least one part of this partition has
chromatic number at least 2, and so contains a copy of $C_3$.  Let
$Y_{Q_1,Q_2}\subseteq V(D)\setminus{Q}$ be the set of vertices seeing
all vertices of $Q_1$ and seen by all vertices of $Q_2$. Observe that
$Y_{Q_1,Q_2}$ is $C_3$-free, otherwise a copy of $C_3$ in
$Y_{Q_1,Q_2}$ together with a copy of $C_3$ in either $Q_1$ or $Q_2$
forms a copy of $C_3\Rightarrow C_3$.  Note that $V(T)\setminus Q$ is
covered by only $2^{|Q|}$ such sets $Y_{Q_1,Q_2}$, and hence $\chi(T)$
is bounded.

Hence we wish to find such a small set of vertices $Q$ with chromatic number 3. 
To this end, we call an arc $uv$ of $T$ \emph{thick} if $N^-(u)\cap N^+(v)$ 
contains a copy of $C_3$. If $T$ has no thick arcs, then 
intuitively $T$ should have simple structure, and thus, bounded chromatic number. 
Suppose that $T$ contains a (not necessarily oriented) triangle $uvw$  where all of the three arcs are thick. 
Then for each of the 
three thick arcs, we take its thickness-certificate (i.e., a copy of $C_3$) and
together with $u,v,w$ we obtain a set $Q$ of at most $12$ vertices. 
It is straightforward to verify that 
$Q$ has chromatic number at least $3$, and thus, by the argument above
$\chi(T)$ is bounded. If $T$ contains no triangle of thick arcs, then for any vertex $v$,
the set of vertices adjacent to $v$ by a thick arc induces a thick-arc-free tournament, 
which, intuitively, should have bounded chromatic number. 
We then easily bound the chromatic number of the sets of 
non-thick in-neighbors and non-thick out-neighbors of $v$, and hence bound the chromatic number of $T$.

The proof of the general case is in the same vein. 
Intuitively, we search for a small set $Q$ with large chromatic number as described above.
We will capture the notion of the set $Q$ with the definition of an object
called an \emph{$r$-mountain}, and the notion of a triangle of thick arcs with
objects called \emph{$(r,s)$-cliques}.
Given a digraph $D$, the formal definitions (which are borrowed from \cite{BCC13}) 
of \emph{$r$-thick-arc}, \emph{$(r,s)$-clique}, and \emph{$r$-mountain} in
$D$ are defined inductively on $r$ as follows. 
Every vertex of $D$ is a {$1$-mountain}. For every $r,s\ge 1$,

\begin{itemize}
\item An arc $e=uv$ of $D$ is \emph{$r$-thick} if $N^-(u)\cap N^+(v)$ contains an $r$-mountain. An $r$-mountain in $N^-(u)\cap N^+(v)$ is a \emph{certificate} of $r$-thickness of $e$. 
\item An \emph{$(r,s)$-clique} of $D$ 
is a set $S\subseteq V(D)$ such that $|S|=s$, and for every distinct vertices $u,v\in S$, 
either $uv$ or $vu$ is an arc that is $r$-thick.
\item Given an $(r,r+1)$-clique $S$ and a certificate $C_{u,v}$ for every distinct $u,v\in S$, 
then the tournament induced on $S\cup(\bigcup_{u,v\in S}C_{u,v})$ is an \emph{$(r+1)$-mountain} of $D$.
\end{itemize}

Note that if a digraph $D$ contains an $(r,r+1)$-clique, then $D$ contains an $(r+1)$-mountain, which is the $(r,r+1)$-clique together with certificates of all $r$-thick arcs of that $(r,r+1)$-clique. Hence, if $D$ contains no $(r+1)$-mountain, then $D$ contains no $(r,r+1)$-clique.

\begin{lemma}[\cite{BCC13}, Lemma~3.3]\label{lem:mountain}
Every $r$-mountain has chromatic number at least $r$, and has at most $(r!)^2$ vertices.
\end{lemma}

Fix two superheroes $H_1$ and $H_2$.

\begin{lemma}\label{lem:mount-min}
Fix $\alpha\ge 2$ and $r,s\ge 1$, suppose that there are $b_0,b_1,b_2$ such that
\begin{itemize}
\item Every $(H_1\Rightarrow H_2)$-free digraph $D$ with $\alpha(D)\le
  \alpha-1$ has $\chi(D) \leq b_0$.
\item Every $(H_1\Rightarrow H_2)$-free digraph $D$ with $\alpha(D)\le
  \alpha$ and containing no $(r,s)$-clique has $\chi(D) \leq b_1$.
\item Every $(H_1\Rightarrow H_2)$-free digraph $D$ with $\alpha(D)\le \alpha$ and containing no $r$-mountain has $\chi(D)\leq b_2$.  
\end{itemize}
Then there is $b_3$ such that every $(H_1\Rightarrow H_2)$-free digraph $D$ with $\alpha(D)\le \alpha$ and containing no $(r,s+1)$-clique has $\chi(D)\leq b_3$.
\end{lemma}

\begin{proof}
A small remark is that the second hypothesis seems redundant since if $D$ contains no $(r,2)$-clique, 
then $D$ contains no $r$-thick arc, and so contains no $r$-mountain. 
However, the second hypothesis is necessary for the case $s$=1.

First, note that since $H_1$ and $H_2$ are superheroes, there is $b_4$
such that every $H_1$-free (or $H_2$-free) digraph $D$ with
$\alpha(D)\le \alpha$ has $\chi(D)\le b_4$.  We first identify all
$r$-thick arcs of $D$. Fix an arbitrary vertex $v$.  Then
$V(D)\setminus{\{v\}}$ can be partitioned into four sets: $N^*$ the
set of neighbors of $v$ that are connected to $v$ by an $r$-thick arc;
$N^-=N^-(v)\setminus{N^*}$; $N^+=N^+(v)\setminus{N^*}$; and $N^o(v)$.

Note that $N^o(v)$ is  $(H_1\Rightarrow H_2)$-free and $\alpha(N^o(v))\le \alpha-1$, 
and so by the first hypothesis, $\chi(N^o(v))\le b_0$. The crucial
fact is that the digraph induced by the set 
$N^*$ does not contain an $(r,s)$-clique; indeed, an
$(r,s)$-clique together with $v$ would form an 
$(r,s+1)$-clique, a contradiction to the fact that $D$ has no $(r,s+1)$-cliques. 
Hence by the second hypothesis, $\chi(N^*)\le b_1$.

\begin{claim}\label{cl:thin}
There is $b_5$ such that either $\chi(N^-)\le b_5$ or $\chi(N^+)\le b_5$.
\end{claim}
\begin{cproof}
Suppose that $\chi(N^-)> b_4$, then $N^-$ contains a copy of $H_1$, say $\hh_1$. Note that
$$N^+=\big(M^+(\hh_1)\cap N^+\big)\cup \big(N^o(\hh_1)\cap N^+\big)\cup \big(N^-(\hh_1)\cap N^+\big).$$
\begin{itemize}
\item If $\chi(M^+(\hh_1)\cap N^+)> b_4$, then $M^+(\hh_1)\cap N^+$ contains a 
copy of $H_2$, say $\hh_2$, and we have $\hh_1\Rightarrow \hh_2$ forming
a copy of $H_1\Rightarrow H_2$, a contradiction. Hence, $\chi(M^+(\hh_1)\cap N^+)\le b_4$

\item For each $u\in \hh_1$, we have $\alpha(N^o(u)\cap N^+)\le \alpha-1$, so $\chi(N^o(u)\cap N^+)\le b_0$.

\item For each $u\in \hh_1$, if $\chi(N^-(u)\cap N^+)\ge b_2$, then $N^-(u)\cap N^+$ contains a $r$-mountain. 
This means that $uv$ is an $r$-thick arc, contradicting $u\notin N^*$. Hence $\chi(N^-(u)\cap N^+)\le b_2$,
for each $u \in \hh_1$.
\end{itemize}
Thus we have, 
\begin{align*}
\chi(N^+)\le &\ \chi(M^+(\hh_1)\cap N^+) + \sum_{u\in \hh_1}\Big(\chi(N^o(u)\cap N^+)+\chi(N^-(u)\cap N^+)\Big)\\
\le &\ b_4+|H_1|(b_0+b_2).
\end{align*}
Set $b_5:= b_4+|H_1|(b_0+b_2)$. We have just shown that if $\chi(N^-)> b_4$, then $\chi(N^+)\le b_5$. Hence either $\chi(N^-)\le b_4\le b_5$ or $\chi(N^+)\le b_5$. This proves the claim.
\end{cproof}

Note that $N^+(v)\subseteq N^+\cup N^*$, 
and so $\chi(N^+(v))\le \chi(N^+)+\chi(N^*)$. Similarly, 
$\chi(N^-(v))\le \chi(N^-)+\chi(N^*)$. Hence for every $v\in V(D)$, 
either $\chi(N^+(v))\le b_5+b_1$ or $\chi(N^-(v))\le b_5+b_1$. Let $R$ be the 
set of all vertices $v\in V(D)$ with $\chi(N^+(v))\le b_5+b_1$ and $B$ be the set 
of all vertices $v\in V(D)$ with $\chi(N^-(v))\le b_5+b_1$. Note that $R\cup B=V(D)$.

Observe that $R$ is a digraph with $\alpha(R)\le \alpha$, and $\chi(N^+_R(v))\le\chi(N^+_D(v))\le b_5+b_1$ for every $v\in R$. Then applying Theorem \ref{thm:local-global-main} to $R$ with $t=b_5+b_1$, there is $b_6$ such that $\chi(R)\le b_6$. Similarly, by Theorem \ref{thm:local-global-main-in}, there is $b_7$ such that $\chi(B)\le b_7$.
Hence $\chi(D)\le\chi(R)+\chi(B)\le b_6+b_7$. Setting $b_3:= b_6 + b_7$
completes the proof of Lemma \ref{lem:mount-min}.
\end{proof}

Recall that if $D$ contains no $(r+1)$-mountain, then $D$ contains no $(r,r+1)$-clique. 
We are now ready to show that digraphs that do not contain an $r$-mountain have bounded chromatic number.

\begin{lemma}\label{lem:mount}
Let $\alpha\ge 2$, and suppose 
that every $(H_1\Rightarrow H_2)$-free digraph $D$ with $\alpha(D)\le
\alpha-1$ has $\chi(D) \leq b_0$ for some $b_0$.
Then for every $r$, there exists
$g_\alpha(r)$ such that every $(H_1\Rightarrow H_2)$-free digraph 
$D$ with $\alpha(D)\le \alpha$ and not containing an $r$-mountain has
$\chi(D)\leq g_\alpha(r)$.
\end{lemma}
\begin{proof}
We proceed by induction on $r$. 
If $D$ contains no $1$-mountain, then $D$ has no vertices, and we can set $g_{{\alpha}}(r):=0$. 
Now suppose by induction that $g_{{\alpha}}({r})$ exists. We will show that  $g_{{\alpha}}({r}+1)$ exists.
First, we claim the following.

\begin{enumerate}[label=(\Alph*)]
\item \label{en:4.2} For every $s$, there exists 
function $g'_{{\alpha}, {r}}(s)$ such that if $D$ contains
no $({r},s)$-clique, then $\chi(D)\le g'_{{\alpha}, {r}}(s)$.
\end{enumerate}

We prove \ref{en:4.2} by induction on $s$. For $s=1$, if $D$ contains no $({r},1)$-clique, then $D$ has no vertex, 
so $g'_{{\alpha}, {r}}(s)=0$.
Suppose, by induction, that $g'_{{\alpha}, {r}}(s)$ exists. Let $D$ be a digraph not containing a
$({r},s+1)$-clique. Applying Lemma \ref{lem:mount-min} 
with  $b_1=g'_{{\alpha},{r}}(s)$ and $b_2=g_{{\alpha}}({r})$, 
we deduce that $g'_{{\alpha}, {r}}(s+1)$ exists. This proves \ref{en:4.2}.

If $D$ contains no $({r}+1)$-mountain, then $D$ contains 
no $({r},{r}+1)$-clique, implying $\chi(D)\le g'_{{\alpha},{r}}({r}+1)$. 
Set $g_{{\alpha}}({r+1}):=g'_{{\alpha},{r}}({r}+1)$. This completes the proof.
\end{proof}

To prove Theorem \ref{lem:HH-main}, it suffices to prove the following lemma.

\begin{lemma}\label{lem:HH}
For every integer $\alpha\ge 1$, there exists
$f(\alpha)$ such that every $(H_1\Rightarrow H_2)$-free 
digraph $D$ with $\alpha(D)\le \alpha$ has $\chi(D)\le f(\alpha)$.
\end{lemma}
\begin{proof}
We proceed by induction on $\alpha$. Since $H_1\Rightarrow H_2$ is a hero (see \cite{BCC13}, Theorem 3.2), 
Lemma \ref{lem:HH} is true for $\alpha =1$. Suppose that Lemma \ref{lem:HH} is true 
for $\alpha-1$ and let $c_0=f(\alpha-1)$. Since both $H_1$ and $H_2$ are superheroes, 
there exists $c_1$ such that if $D$ is any $H_1$-free or $H_2$-free digraph, then $\chi(D) \leq c_1-1$.
Let $D$ be a $(H_1\Rightarrow H_2)$-free digraph with vertex set $V$ and $\alpha(D)\le \alpha$. 
If $D$ does not contain a $(c_0+2c_1)$-mountain, then by applying Lemma \ref{lem:mount} to $D$ 
with $b_0=c_0$, there is $c_2$ such that $\chi(D)\le c_2$. 
Thus, it remains to consider the case that $D$ contains a $(c_0+2c_1)$-mountain. 

\begin{claim}\label{cl:max-mount}
There is $c_3$ such that if $D$ contains a $(c_0+2c_1)$-mountain, then $\chi(D)\le c_3$.
\end{claim}
\begin{cproof}
Let $Q$ be a $(c_0+2c_1)$-mountain of $D$. Then by Lemma \ref{lem:mountain}, 
$|Q|\le ((c_0+2c_1)!)^2$ and $\chi(Q)\ge c_0+2c_1$. For every partition $Q$ into three 
sets $Q_0,Q_1,Q_2$, let $Y_{Q_0,Q_1,Q_2}$ be the set of vertices $v\in V\setminus {Q}$ 
such that $Q_0\subseteq N^o(v), Q_1\subseteq N^+(v)$, and $Q_2\subseteq N^-(v)$. Note that 
for every vertex $v\in V\setminus{Q}$, there always exists a partition of $Q$ into 
some sets $Q_0,Q_1,Q_2$ such that $v$ is non-adjacent with every vertex in $Q_0$, sees every 
vertex in $Q_1$ and is seen by every vertex in $Q_2$.
Hence, $V\setminus{Q}$ can be written as
the union of all possible $Y_{Q_0,Q_1,Q_2}$. There are $3^{|Q|}$ sets $Y_{Q_0,Q_1,Q_2}$.

\begin{enumerate}[label=(\Alph*)]\setcounter{enumi}{1}
\item \label{en:4.1} $\chi(Y_{Q_0,Q_1,Q_2})\le c_1$ for every partition $(Q_0,Q_1,Q_2)$ of $Q$.
\end{enumerate}
Indeed, if $Y_{Q_0,Q_1,Q_2}=\emptyset$, then \ref{en:4.1} clearly
holds. Otherwise, $Q_0\subseteq N^o(v)$ for any $v\in
Y_{Q_0,Q_1,Q_2}$.  Note that the digraph $Q_0$ is $(H_1\Rightarrow
H_2)$-free and $\alpha(Q_0)\le \alpha-1$, and so by induction
hypothesis, $\chi(Q_0)\le c_0.$ This gives $\chi(Q_1\cup Q_2)\ge
\chi(Q)-\chi(Q_0)\geq 2c_1$, implying that either $\chi(Q_1)\ge c_1$
or $\chi(Q_2)\ge c_1$.  If $\chi(Q_1)\ge c_1$, then $Q_1$ contains a
copy of $H_2$, say $\hh_2$. If $\chi(Y_{Q_0,Q_1,Q_2})\ge c_1$, then
$Y_{Q_0,Q_1,Q_2}$ contains a copy of $H_1$, say $\hh_1$.  Then
$\hh_1\Rightarrow \hh_2$ forms a copy of $H_1\Rightarrow H_2$, a
contradiction.  Hence, $\chi(Y_{Q_0,Q_1,Q_2})\le c_1$. A similar
argument establishes the case $\chi(Q_2)\ge c_1$.  This proves
\ref{en:4.1}.

Hence 
\begin{align*}
\chi(D)\le&\ \chi(Q)+\chi(V\setminus{Q})\\
\le&\ |Q|+\chi(\bigcup_{(Q_0,Q_1,Q_2)} Y_{Q_0,Q_1,Q_2})\\
\le&\ |Q|+\sum_{(Q_0,Q_1,Q_2)} {\chi(Y_{Q_0,Q_1,Q_2})}\\
\le&\ ((c_0+2c_1)!)^2+3^{((c_0+2c_1)!)^2}c_1.
\end{align*}
Set $c_3:=((c_0+2c_1)!)^2+3^{((c_0+2b_1)!)^2}c_1$. This completes proof of the claim.
\end{cproof}

Hence $\chi(D)\le \max(c_2,c_3)$. Setting
$f(\alpha):=\max(c_2,c_3)$ completes the proof of
Lemma \ref{lem:HH}, thus proving Theorem \ref{lem:HH-main}.
\end{proof}


\section{Proof of Theorem~\ref{lem:h1k-main}} \label{sec:H1K}

It is proved in \cite{S59} that for each integer $k\ge 1$, 
every tournament with at least $2^{k-1}$ vertices contains a copy of $T_k$. 
Let ${\cal R}(a,b)$ be the Ramsey number of $(a,b)$ (i.e., the smallest $n$
such that any graph on $n$ vertices either contains an independent set of order $a$
or a clique of order $b$).
\begin{proposition}\label{lem:k-main}
For each integer $k\ge 1$, every $T_k$-free digraph $D$ with $\alpha(D)\le \alpha$ has at most ${\cal R}(\alpha+1,2^{k-1})$ vertices.
\end{proposition}
\begin{proof}
Suppose for a contradiction that there is a $T_k$-free digraph $D$ with at least ${\cal R}(\alpha+1,2^{k-1})$ vertices. Then the underlying graph of $D$ contains either an independent set of size $\alpha+1$ or a clique of size $2^{k-1}$. The former case is impossible since $\alpha(D)\le \alpha$. Thus $D$ contains a tournament of size $2^{k-1}$, and hence contains a copy of $T_k$, a contradiction.
\end{proof}

Recall that Theorem \ref{lem:h1k-main} states
that if $H$ is a superhero, then so are $\Delta(H,T_k,T_1)$ 
and $\Delta(H,T_1,T_k)$ for any integer $k\ge 1$.
We will prove that if $H$ is a superhero, then so is $\Delta(H,T_k,T_1)$ for any $k\ge 1$. This is sufficient. Indeed, if $H$ is a superhero,
then so is $H_{rev}$, the digraph obtained from $H$ by reversing all its arcs.
Thus, $\Delta(H_{rev},T_k,T_1)_{rev} = \Delta(H,T_1,T_k)$ is also a superhero.

\begin{theorem}\label{thm:h1k-main2}
For every superhero $H$ and every pair of integers $k,\alpha \ge 1$, there is a number
$f(H,k,\alpha)$ such that every $\Delta(H,T_k,T_1)$-free digraph
$D$ with $\alpha(D)\le\alpha$ has $\chi(D)\le f(H,k,\alpha)$.
\end{theorem}
 
The idea of the proof of Theorem \ref{thm:h1k-main2} is as
follows. Fix a large number $c$ and call a subset $B$ of $V(D)$ such that
$\chi(B)=c$ a \emph{bag}.  We aim to find a longest chain of disjoint
bags $B_1,\dots,B_t$ in $V(D)$, together with a partition
$V(D)\setminus {\bigcup B_i}$ into sets we call \emph{zones}
$Z_0,\dots,Z_t$ such that there is no \emph{backward arc} in $D$
(where $uv$ is a backward arc if the bag or zone containing $u$ has
higher index than the bag or zone containing $v$). Then, using
maximality of $t$, we show that the chromatic number of every zone is
bounded.  Once proving this, we observe that all $B_i$ and $Z_i$ have
bounded chromatic number and since $D$ has no backward arc, $D$ has
bounded chromatic number.  However, the requirement that there is no
backward arc in $D$ is too strong and so we will have to slightly
relax it.  In doing so, we will need to allow some backward arcs, but
we will want to do so in a very controlled manner. This leads us to
the following definitions.

For an integer $c\ge 1$ and a $\Delta(H,T_k,T_1)$-free digraph $D$
with $\alpha(D)\le \alpha$, a set $B\subseteq V(D)$ such that $\chi(B)=c$
is called a \emph{$c$-bag}.  A family of pairwise disjoint $c$-bags
$B_1,\dots,B_t$ is a \emph{$c$-bag-chain} if for every $i$ and every
$v\in B_i$, we have $\chi(N^+(v)\cap B_{i-1})\le c_1$ and
$\chi(N^-(v)\cap B_{i+1})\le c_1$, where $c_1$ is a fixed number
satisfying
\begin{itemize}
\item $c_1 \ge \mathcal{R}(\alpha+1,2^{k-1})$, and
\item $c_1 \ge \chi(D)$ for every $H$-free digraph $D$ with $\alpha(D)\le \alpha$.
\end{itemize}
Since $H$ is a superhero, such $c_1$ clearly exists. Note that every $T_k$-free digraph $D$ with $\alpha(D)\le \alpha$ has at most $c_1$ vertices 
by Proposition~\ref{lem:k-main}, and so has chromatic number at most $c_1$. 

\begin{proof}[Proof of Theorem~\ref{thm:h1k-main2}]
We proceed by induction on $\alpha$.  Since $\Delta(H,T_k,T_1)$ is a
hero (see \cite{BCC13}, Theorem~4.1), the theorem holds for $\alpha
=1$. Suppose that Theorem~\ref{thm:h1k-main2} is true for $\alpha-1$
and let $c_0=f(H,k,\alpha-1)$. Let $D$ be a $\Delta(H,T_k,T_1)$-free
digraph with vertex set $V$ and $\alpha(D)\le\alpha$.  An important
observation is that for every $v\in V$, the digraph $N^o(v)$ is
$\Delta(H,T_k,T_1)$-free and has independence number at most
$\alpha-1$, and hence
\begin{equation}\label{eq:no}
\chi(N^o(v))\le f(H,k,\alpha-1)=c_0. 
\end{equation}
We also would like to recall some useful formulas. For every $v\in
X\subseteq V$,
\begin{equation}\label{eq:4.set-vertex}
\chi(N^+(X))\le \sum_{v\in X}\chi(N^+(v)) \text{ and } \chi(N^o(X))\le
\sum_{v\in X}\chi(N^o(v)),
\end{equation}
and if $Y\subseteq V$ and $Y\cap X=\emptyset$, then (recalling that
$M^+(X)$ is the set of vertices seen by all vertices of $X$)
\begin{equation}\label{eq:4.set}
Y= (M^+(X)\cap Y)\cup \big((N^-(X)\cup N^o(X))\cap Y\big).
\end{equation}

Let $h = |H|$ and set $c:=2(c_0+c_1)(h+k)$. Let us assume that
$B_1,\dots,B_t$ is a $c$-bag-chain of $D$ with $t$ as large as
possible. In the proof of this theorem, we drop prefix $c$- of $c$-bag
and $c$-bag-chain for convenience. By definition of bag-chain, every
bag has few backward arcs with bags preceding or succeeding it. In
the following claim, we show that in fact every bag has few backward
arcs with any other bag.

\begin{claim}\label{cl:h1k-bag1}
For every $i$ and $v\in B_i$, and for every $r>0$, 
\begin{enumerate}[label=(\alph*)]
\item \label{en:5.1} $\chi(N^+(v)\cap B_{i-r}) \le c_1$,  and
\item \label{en:5.2} $\chi(N^-(v)\cap B_{i+r})\le c_1$.
\end{enumerate}
\end{claim}
\begin{cproof}
We proceed by induction on $r$. For $r=1$, both \ref{en:5.1} and
\ref{en:5.2} holds by definition of bag-chain. Suppose that both
statements are true for $r-1$. We now prove \ref{en:5.1} for
$r$. Suppose for a contradiction that there is $v$ in some $B_i$ such
that $\chi(N^+(v)\cap B_{i-r}) > c_1$. Then $N^+(v)\cap B_{i-r}$ has
a copy of $H$, say $\hat{H}$.  Then by applying \eqref{eq:4.set} we
have
$$B_{i-1}=\Big(M^+(\hh)\cap B_{i-1} \Big)\cup \Big(\big(N^-(\hh)\cup N^o(\hh)\big)\cap B_{i-1} \Big),$$
and 
$$B_{i-1}=\Big(N^-(v)\cap B_{i-1} \Big)
\cup \Big(\big(N^+(v)\cup N^o(v)\big)\cap B_{i-1} \Big).$$
Thus (by using the fact that if $A=B\cup C=B'\cup C'$, then $A=(B\cap B')\cup C\cup C'$) we have 
\begin{equation}\label{eq:1}
\begin{split}
B_{i-1}=&\Big(M^+(\hh)\cap N^-(v)\cap B_{i-1} \Big)\cup
\Big(\big(N^-(\hh)\cup N^o(\hh)\big)\cap B_{i-1} \Big) \\ &
\ \ \ \ \cup \Big(\big(N^+(v)\cup N^o(v)\big)\cap B_{i-1} \Big).
\end{split}
\end{equation}

For each $x\in \hat{H}$, by (\ref{eq:no}) we have $\chi(N^o(x)\cap
B_{i-1})\le c_0$, and by induction hypothesis of \ref{en:5.2} applied
to $x$ and $r-1$, we have $\chi(N^-(x)\cap B_{i-1})\le c_1$. We also
have $\chi(N^o(v)\cap B_{i-1})\le c_0$ by (\ref{eq:no}) and
$\chi(N^+(v)\cap B_{i-1})\le c_1$ by definition of a
bag-chain. Combining with (\ref{eq:1}) and \eqref{eq:4.set-vertex} we
have
\begin{align*}
\chi(M^+(\hat{H})\cap N^-(v)\cap B_{i-1})\ge & \ \chi(B_{i-1})-\sum_{x\in \hat{H}\cup\{v\}}\chi(N^o(x)\cap B_{i-1})\\
&- \sum_{x\in \hat{H}}\chi(N^-(x)\cap B_{i-1})-\chi(N^+(v)\cap B_{i-1}),\\
\ge& \ 2 (c_0+c_1)(h+k)-c_0(h+1)-c_1h-c_1,\\
> & \ c_1.
\end{align*}
Then there exists a copy of $T_k$ in $ M^+(\hh)\cap N^-(v)\cap B_{i-1}$, say $\hat{T_k}$.  
Note that by construction, we have all arcs from $\hh$ to $\hk$, from $\hk$ to $v$ and from $v$ to $\hh$.
Then $\Delta(\hat{H}, \hat{T_k},v)$ forms a copy of $\Delta(H,T_k,T_1)$, a contradiction.

The proof of \ref{en:5.2} for $r$ is similar but not symmetric. In order to obtain a copy of $\Delta(H,T_k,T_1)$, we first get a copy of $T_k$ in $B_{i+r}$, and then a copy of $H$ in $B_{i+1}$. 
This proves the claim.
\end{cproof}

We next prove a stronger statement that every bag has few backward
arcs with the union of all other bags preceding or succeeding it.
Let $B_{i,j}=\bigcup_{s=i}^{j}B_{s}$ for any $1\le i\le j\le t$. (If
$i<1$ or $j>t$ or $j<i$, we set $B_{i,j}=\emptyset$.)

\begin{claim}\label{cl:h1k-bag}
For every $i$ and $v\in B_i$, 
\begin{itemize}
\item $\chi(N^+(v)\cap B_{1,i-2})\le c_1$, and 
\item $\chi(N^-(v)\cap B_{i+2,t}))\le c_1$.
\end{itemize}
\end{claim}
\begin{cproof}
We repeat the same argument as in the proof of Claim
\ref{cl:h1k-bag1}.  Towards a contradiction, suppose that the first
statement is false: there is $v$ in some $B_i$ such that
$\chi(N^+(v)\cap B_{1,i-2}) > c_1$. Then $N^+(v)\cap B_{1,i-2}$ has a
copy of $H$, say $\hat{H}$. For each $x\in \hat{H}$, we have
$\chi(N^o(x)\cap B_{i-1})\le c_0$ by (\ref{eq:no}), and
$\chi(N^-(x)\cap B_{i-1})\le c_1$ by Claim \ref{cl:h1k-bag1}.  We also
have $\chi(N^o(v)\cap B_{i-1})\le c_0$ and $\chi(N^+(v)\cap
B_{i-1})\le c_1$. Thus by the same computation as in Claim
\ref{cl:h1k-bag1}, we obtain a copy of $T_k$ in $B_{i-1}$ and reach
the contradiction. The proof of the second statement is similar.
\end{cproof}

From Claim \ref{cl:h1k-bag} we have the following immediate corollary.

\begin{claim}\label{cl:bag2}
For every $i$ and $v\in B_i$,
\begin{itemize}
\item $\chi(N^+(v)\cap B_{1,i-1})\le 2c_1$.
\item $\chi(N^-(v)\cap B_{i+1,t})\le 2c_1$. 
\end{itemize}
\end{claim}

We now show that the union of all bags has bounded chromatic number. 
We note that in the following proof, we will use only two 
hypotheses: Claim \ref{cl:bag2} and that $\chi(B_i)$ is bounded for every $i$. 
We will re-use the arguments in this proof for subsequent claims.

\begin{claim}\label{cl:hero4.1}
$\chi(B_{1,t})\le 8c(c_1+1)$.
\end{claim}
\begin{cproof}
An arc $uv$ with $u\in B_j, v\in B_i,$ and $j>i$ is called a
\emph{backarc} with \emph{span} $j-i$.  For every $i$ and every $v\in
B_i$, if $|N^-(v)\cap B_{i+1,t}|< c_1$, let $F_v:=N^-(v)\cap
B_{i+1,t}$.  If $|N^-(v)\cap B_{i+1,t}|\ge c_1$, let $F_v\subseteq
N^-(v)\cap B_{i+1,t}$ consist of $c_1$ vertices whose backarcs to
$v$ have largest possible spans. Let us show that
\begin{enumerate}[label=(\Alph*)]\setcounter{enumi}{0}
\item \label{en:5.3} For every backarc $uv$ with $u\notin F_v$, if $u\in B_j$ and $v\in B_i$ then $\chi(B_{i,j})\le 4c$.
\end{enumerate}
If $j=i+1$, then \ref{en:5.3} clearly holds since $\chi(B_{i,i+1})\le \chi(B_{i})+\chi(B_{i+1})\le 2c$. Thus we may suppose that $j\ge i+2$. 
Since $u\in N^-(v)\cap B_{i+1,t}$ but $u\notin F_v$, it follows from 
the definition of $F_v$ that $|F_v|=c_1$. Hence $|F_v\cup\{u\}|=c_1+1$. Then there is a copy of $T_k$ in $F_v\cup\{u\}$, say $\hk$. 
Note that $\hk\subseteq B_{j,t}$.

We have a formula similar to \eqref{eq:1}:
\begin{align*}
B_{i+1,j-1}=&\Big(M^-(\hk)\cap N^+(v)\cap B_{i+1,j-1}\Big)\cup \Big(\big(N^+(\hk)\cup N^o(\hk)\big)\cap B_{i+1,j-1} \Big)\cup \\
& \ \ \ \ \cup \Big(\big(N^-(v)\cup N^o(v)\big)\cap B_{i+1,j-1} \Big).
\end{align*}
For each $x\in \hk\cup\{v\}$, we have $\chi(N^o(x)\cap B_{i+1,j-1})\le c_0$. Note also that from Claim \ref{cl:bag2}, we have  
$\chi(N^+(x)\cap B_{i+1,j-1})\le 2c_1$ for every $x\in \hk$ and $\chi(N^-(v)\cap B_{i+1,j-1})\le 2c_1$. 
Furthermore, if $\chi(M^-(\hk)\cap N^+(v)\cap B_{i+1,j-1})> c_1$, then $M^-(\hk)\cap N^+(v)\cap B_{i+1,j-1}$ contains a copy of $H$, say $\hh$. Then $\Delta(\hh,\hk,v)$ forms a copy of $\Delta(H,T_k,T_1)$, a contradiction.
Hence 
\begin{align*}
\chi(B_{i+1,j-1})\le&\  \chi\Big(M^-(\hk)\cap N^+(v)\cap B_{i+1,j-1}\Big)+\sum_{x\in \hat{T_k}\cup\{v\}}\chi\Big(N^o(x)\cap B_{i+1,j-1}\Big)\\
&\  +\sum_{x\in \hat{T_k}}\chi\Big( N^+(x)\cap B_{i+1,j-1}\Big)+
\chi\Big(N^-(v)\cap B_{i+1,j-1}\Big)\\
\le &\ c_1+c_0(k+1)+2c_1k+2c_1\\
\le &\ 3(c_1+c_0)(k+1)\le 2c,
\end{align*}
since $c=2(c_0+c_1)(h+k)$.
This gives $\chi(B_{i,j})\le \chi(B_i)+\chi(B_{i+1,j-1})+\chi(B_j)\le 4c$, which proves \ref{en:5.3}.

Now let $G$ be the undirected graph with vertex set $B_{1,t}$ and
$uv\in E(G)$ if $u\in F_v$ or $v\in F_u$.  Then $B_i$ is a stable set
in $G$ for every $i$.  We now color the vertices of $G$ by $c_1+1$
colors as follows.  First, color all $B_t$ properly by color
1. Suppose that we have already colored $B_{i+1},\dots,B_t$. Every
vertex $v$ in $B_i$ is incident (in $G$) with at most $c_1$ vertices
in $B_{i+1,t}$ (those belonging to $F_v$) and independent (in $G$)
with all other vertices in $B_i$, so we can always properly color $v$,
and so properly color $B_i$.  When the process of coloring ends, we
obtain a partition of $B_{1,t}$ into $c_1+1$ sets of colors, say
$X_1,\dots,X_{c_1+1}$, where each $X_s$ is a stable set in $G$. We now
claim that the chromatic number of the graph induced by $X_s$ is
small.

\begin{enumerate}[label=(\Alph*)]\setcounter{enumi}{1}
\item \label{en:5.4} $\chi(X_s)\le 8c$ for every $1\le s\le c_1+1$.
\end{enumerate}

To prove \ref{en:5.4}, we define a sequence of indices $i_1,i_2,\dots$
inductively as follows. Let $i_1=1$, and for every $r\ge 1$, let
$i_{r+1}>i_r$ be the smallest index such that
$\chi(B_{i_{r},i_{r+1}})> 4c$. The sequence ends by $i_\ell$ with
$\chi(B_{i_{\ell},t})\le 4c$ (i.e., there is no $i_{\ell+1}$
satisfying the condition).  Set $A_r:=B_{i_{r},i_{r+1}-1}$ for every
$1\le r\le \ell-1$ and $A_\ell:=B_{i_{\ell},t}$. Then
$B_{1,t}=\bigcup_{r=1}^{\ell}A_r$, and by definition of the sequence,
$\chi(B_{i_{r},i_{r+1}-1})\le 4c$ for every $1\le r\le \ell-1$. In
other words, for every $1\le r \le \ell$,
\begin{equation}\label{en:5.5} 
\chi(A_r)\le 4c.
\end{equation}

Suppose that there is a backarc $uv$ with $u\notin F_v$ and $u\in
A_{r},v\in A_{r'}$, where $r\ge r'+2$.  Suppose that $u\in B_{j}$ and
$v\in B_{j'}$.  Then $j\ge i_{r}$ since $B_{j}\subseteq
A_r=B_{i_{r},i_{r+1}-1}$, and $j'< i_{r'+1}$ since $B_{j'}\subseteq
A_{r'}=B_{i_{r'},i_{r'+1}-1}$, and so $j'< i_{r-1}$ since $r'+1\le
r-1$.  Also note that $\chi(B_{i_{r-1},i_{r}})> 4c$ by definition of
the sequence. Thus we have $$\chi(B_{j',j})\ge
\chi(B_{i_{r-1},i_{r}})\ge \chi(B_{i_{r},i_{r+1}})> 4c,$$ which
contradicts \ref{en:5.3} in that $\chi(B_{j',j})\le 4c$ if $u\notin
F_v$. This shows that for any $r\ge r'+2$, there is no backarc $uv$
with $u\in A_r,v\in A_{r'}$ and $u\notin F_v$.

Now fix an arbitrary $X_s$ and let $X_{s,r}=X_s\cap A_r$ for every
$1\le r\le \ell$.  Observe that if $uv$ is a backarc with $u,v\in
X_s$, then $u\notin F_v$ (otherwise, $u\in F_v$, so $uv\in G$ and so
$X_s$ is not stable in $G$, a contradiction). Hence combining with the
observation in the paragraph above, we have that there is no backarc
from $X_{s,r}$ to $X_{s,r'}$ for any $r,r'$ with $r\ge r'+2$; in other
words, $X_{s,r'}\to X_{s,r}$ for any $r,r'$ with $r\ge r'+2$.  Recall
from \eqref{en:5.5} that $\chi(X_{s,r})\le 4c$. Thus
$\chi\big(\bigcup_{r\ge1}X_{s,2r-1}\big)\le 4c$ and
$\chi\big(\bigcup_{r\ge 1}X_{s,2r}\big)\le 4c$. This gives
$\chi(X_s)=\chi\big(\bigcup_{r\ge 1}X_{s,r}\big)\le 8c$, which proves
\ref{en:5.4}.

Hence $$\chi(B_{1,t})\le \sum_{s=1}^{c_1}\chi(X_s)\le 8c(c_1+1).$$ This completes the proof of Claim \ref{cl:hero4.1}.
\end{cproof}

We now turn our attention to the vertices not in the bags.  We
partition $V\setminus{B_{1,t}}$ into sets $Z_0,\dots,Z_t$ called
\emph{zones} as follows. For every $x\in V\setminus{B_{1,t}}$, if $i$
is the largest index such that $\chi(N^-(x)\cap B_i) > c_1$, then
$x\in Z_i$; otherwise, $x\in Z_0$.  We first show that every $Z_i$ has
few backward arcs with any $B_j$ sufficiently far from it.

\begin{claim}\label{cl:h1k-zone-wit}
For every $i$ and $v\in Z_i$,
\begin{itemize}
\item $\chi(N^-(v)\cap B_{i+r})\le c_1$ for every $r\ge 1$, and
\item $\chi(N^+(v)\cap B_{i-r})\le c_1$ for every $r\ge 2$. 
\end{itemize}
\end{claim}

\begin{cproof}
The former inequality is obvious by the partition criterion. For the
latter, the proof follows the same idea as that of Claim
\ref{cl:h1k-bag1}, but is a bit more involved.  Suppose for a
contradiction that $\chi(N^+(v)\cap B_{i-r}) > c_1$ for some $r\ge
2$. Then there is a copy of $H$ in $N^+(v)\cap B_{i-r}$, say
$\hh$. Since $\chi(N^-(v)\cap B_i)> c_1$ by partition criterion,
there is a copy of $T_k$ in $N^-(v)\cap B_i$, say $\hk$.  Then by
applying \eqref{eq:4.set} we have
$$B_{i-1}=\Big(M^+(\hh)\cap B_{i-1} \Big)\cup \Big(\big(N^-(\hh)\cup N^o(\hh)\big)\cap B_{i-1} \Big),$$
$$B_{i-1}=\Big(M^-(\hk)\cap B_{i-1} \Big)\cup \Big(\big(N^+(\hk)\cup N^o(\hk)\big)\cap B_{i-1} \Big),$$
and
$$B_{i-1}=\Big(\big(N^+(v)\cup N^-(v)\big)\cap B_{i-1} \Big)\cup \Big(N^o(v)\cap B_{i-1} \Big).$$
Let $R=M^+(\hat{H})\cap M^-(\hk)\cap \big(N^+(v)\cup N^-(v)\big)\cap B_{i-1}$. Then we have 
\begin{align*}
B_{i-1}=R\cup \Big(\big(&\ N^-(\hh)\cup N^o(\hh)\big)\cap B_{i-1} \Big)\cup \\
& \cup \Big(\big(N^+(\hk)\cup N^o(\hk)\big)\cap B_{i-1} \Big) \cup \big(N^o(v)\cap B_{i-1} \big).
\end{align*}

For each $x\in \hh\cup \hk\cup\{v\}$, we have $\chi(N^o(x)\cap B_{i-1})\le c_0$. By Claim~\ref{cl:h1k-bag1}, we have $\chi(N^-(x)\cap B_{i-1})\le c_1$ for each $x\in \hh$ and $\chi(N^+(x)\cap B_{i-1})\le c_1$ for each $x\in \hk$. 
Then
\begin{align*}
\chi(R)\ge & \ \chi(B_{i-1})-\sum_{x\in \hat{H}\cup\hk\cup\{v\}}\chi(N^o(x)\cap B_{i-1})\\
&- \sum_{x\in \hat{H}}\chi(N^-(x)\cap B_{i-1})-\sum_{x\in \hk}\chi(N^+(x)\cap B_{i-1})\\
\ge& \ 2(c_0+c_1)(h+k)-c_0(h+k+1)-c_1h-c_1k\\
>& \ 2c_1.
\end{align*}
Let $R_1=R\cap N^+(v)$ and $R_2=R\cap N^-(v)$.
Note that $R=R_1\cup R_2$, and so either $\chi(R_1)> c_1$ or $\chi(R_2)> c_1$. If $\chi(R_1)> c_1$, there is a copy of $H$ in $R_1$, say $\hh'$. Then $\Delta(\hh',\hk,v)$ forms a copy of $\Delta(H,T_k,T_1)$, a contradiction. If $\chi(R_2)> c_1$, there is a copy of $T_k$ in $R_2$, say $\hk'$. Then $\Delta(\hh,\hk',v)$ forms a copy of $\Delta(H,T_k,T_1)$, a contradiction again. This proves the claim. 
\end{cproof}

Claim \ref{cl:h1k-zone-wit} shows that every zone has few
backward arcs with every bag sufficiently far from it. The next claim
shows the converse (i.e., every bag has few backward arcs with every
zone sufficiently far from it).

\begin{claim}\label{cl:h1k-zone-wit2}
For every $i$ and $v\in B_i$,
\begin{itemize}
\item $\chi(N^+(v)\cap Z_{i-r})\le c_1$ for every $r\ge 2$, and
\item $\chi(N^-(v)\cap Z_{i+r})\le c_1$ for every $r\ge 3$.
\end{itemize}
\end{claim}
\begin{cproof}
We repeat the argument in the proof of Claim \ref{cl:h1k-bag1}. 
Suppose for a contradiction that the first statement is false: there
is $v$ in some $B_i$ and $r\ge 2$ such that $\chi(N^+(v)\cap Z_{i-r}) > c_1$. Then $N^+(v)\cap Z_{i-r}$ has a copy of $H$, say $\hat{H}$. For each $x\in \hat{H}$, we have $\chi(N^o(x)\cap B_{i-1})\le c_0$ by (\ref{eq:no}), and $\chi(N^-(x)\cap B_{i-1})\le c_1$ by the first inequality in Claim \ref{cl:h1k-zone-wit}.
We also have $\chi(N^o(v)\cap B_{i-1})\le c_0$ and $\chi(N^+(v)\cap
B_{i-1})\le c_1$ by definition of a bag-chain. Thus by the same computation as in Claim \ref{cl:h1k-bag1}, we obtain a copy of $T_k$ in $B_{i-1}$ and reach the contradiction. The proof of the second statement is similar.
\end{cproof}

The next claim is a counterpart of Claim \ref{cl:bag2} for zones, and will be used to bound the chromatic number of the union of zones.

\begin{claim}\label{cl:h1k-zone1}
For every $i$ and $v\in Z_i$, 
\begin{itemize}
\item $\chi(N^+(v)\cap \bigcup_{s=0}^{i-3}Z_{s})\le c_1$, and
\item $\chi(N^-(v)\cap \bigcup_{s=i+3}^{t}Z_{s})\le c_1$.
\end{itemize}
\end{claim}

\begin{cproof}
Suppose for a contradiction that $\chi(N^+(v) \cap
\bigcup_{s=0}^{i-3}Z_{s})> c_1$, then there is a copy of $H$ in
$N^+(v) \cap \bigcup_{s=0}^{i-3}Z_{s}$, say $\hh$.  For each $x\in
\hat{H}$, we have $\chi(N^o(x) \cap B_{i-2})\le c_0$ by (\ref{eq:no}),
and $\chi(N^-(x) \cap B_{i-2})\le c_1$ by the first inequality in Claim
\ref{cl:h1k-zone-wit}.  We also have $\chi(N^o(v) \cap B_{i-2})\le c_0$
and $\chi(N^+(v) \cap B_{i-2})\le c_1$ by the second inequality in
Claim \ref{cl:h1k-zone-wit}.  Thus by the same computation as in Claim
\ref{cl:h1k-bag1}, we obtain that $ \chi (M^{+}(\hat{H}) \cap N^{-}(v)
\cap B_{i-2}) > c_1$ implying that there is a copy of $T_k$ in
$M^{+}(\hat{H}) \cap N^{-}(v) \cap B_{i-2}$, reaching a contradiction.
The proof of the second statement is similar, where we use $B_{i+1}$
instead of $B_{i-2}$.
\end{cproof}

In order to bound the chromatic number of the union of zones, we also need that each zone has bounded chromatic number. We will prove this by employing the assumption that the bag-chain $B_1,\dots,B_t$ 
is of maximal length.
\begin{claim} \label{cl: 6-chain}
No zone $Z_i$ contains a bag-chain of length 6.
 
\end{claim}

\begin{cproof}

Suppose that some $Z_i$ contains a bag-chain of length 6, say $Y_1,\dots,Y_6$. Note that we have
\begin{itemize}
\item $\chi(N^+(v)\cap B_{i-3})\le c_1$ for every $v\in Y_1$ by Claim~\ref{cl:h1k-zone-wit} and $\chi(N^-(v)\cap Y_{1})\le c_1$ for every $v\in B_{i-3}$ by Claim~\ref{cl:h1k-zone-wit2}; and
\item $\chi(N^-(v)\cap B_{i+3})\le c_1$ for every $v\in Y_6$ by Claim~\ref{cl:h1k-zone-wit} and $\chi(N^+(v)\cap Y_6)\le c_1$ for every $v\in B_{i+3}$ by Claim~\ref{cl:h1k-zone-wit2}.
\end{itemize}
Thus by definition of a bag-chain,
$B_1,\dots,B_{i-3},Y_1,Y_2,\dots,Y_6,B_{i+3},\dots,B_t$ is a bag-chain
of length $t+1$, which contradicts the maximality of $t$. Hence no
$Z_i$ contains a bag-chain of length 6.
\end{cproof}

\begin{lemma} \label{zone:bounded}
There is $c'$ such that if a $\Delta(H,T_k,T_1)$-free digraph $D'$ with $\alpha(D')\le \alpha$ does not contain any $c$-bag-chain of length $6$, then $\chi(D')\le c'$. 
\end{lemma}

We defer the proof of Lemma \ref{zone:bounded} for now.

We now show that this is sufficient to prove the theorem. To do so, we
group zones by indices modulo 3 and follow a similar argument as in
Claim \ref{cl:hero4.1}.  Fix $0\le s\le 2$, and for every $i\equiv s
\mod 3$, $1\le i\le t$, let $Z^s_{\lfloor i/3\rfloor}:=Z_i$. 
By Claim \ref{cl:h1k-zone1}, for every $j$ and for every
$v\in Z^s_j$, we have

\begin{itemize}
\item $\chi(N^+(v)\cap \bigcup_{r< j}Z^s_r)\le c_1$, and
\item $\chi(N^-(v)\cap \bigcup_{r>j}Z^s_{r})\le c_1.$
\end{itemize} 

By applying Lemma \ref{zone:bounded} on $D'=Z_i$ and using Claim
\ref{cl: 6-chain}, it follows that $\chi(Z_i)\le c'$ for every $i$.
Thus, $\chi(Z^s_j)\le c'$ for every $j$. We can repeat exactly 
the argument of Claim \ref{cl:hero4.1} (with $c$ replaced by $c'$
and by using Claim \ref{cl:h1k-zone1} instead of Claim \ref{cl:bag2}) to deduce 
$$\chi\Big(\bigcup_{r\ge 0}Z^s_{r}\Big)\le 8c'c_1$$ for every
$s=0,1,2$.  (Remark: we may assume that $c'> c$, which is necessary
for the argument at the end of the proof of Property \ref{en:5.3} that
$3(c_1+c_0)(k+1)\le 2c < 2c'$.)

Hence
$$\chi(\bigcup_{s=0}^{t}Z_s)\le \sum_{s=0}^{2}\chi\Big(\bigcup_{r\ge 0}Z^s_{r}\Big)\le 24c'c_1.$$

From Claims \ref{cl:hero4.1} and the above inequality, we
have $$\chi(D)\le \chi(B_{1,t})+ \chi( \cup_{i=0}^{t} Z_i)\le
8c(c_1+1) + 24c'c_1,$$
which proves Theorem~\ref{thm:h1k-main2}.
\end{proof}

\subsection{Proof of Lemma \ref{zone:bounded}}

Lemma \ref{zone:bounded} asserts that having no bag-chain of length 6
is enough to force a digraph (with bounded independence number and
$\Delta(H,T_k,T_1)$-free) to have bounded chromatic number. It is
trivial by definition that having no bag-chain of length 1 forces
bounded chromatic number. In the next lemma, we show that having no bag-chain
of length 2 can force bounded chromatic number, which contains most of
the difficulty of the proof of Lemma \ref{zone:bounded}.  In the
following lemma, $c,c_0,c_1,h,k$ are the values as in the proof of
Theorem \ref{thm:h1k-main2}. Recall that $c= 2(c_0+c_1)(h+k)$.

\begin{lemma}\label{lem:h1k-2chain}
For every 
$d\ge c$,
there is $g(d)$ such that if a $\Delta(H,T_k,T_1)$-free digraph $D$ 
with $\alpha(D)\le \alpha$ contains no $d$-bag-chain of length $2$, then $\chi(D)\le g(d)$.
\end{lemma}

\begin{proof}

Let $J$ be the tournament $H\Rightarrow T_k$. By Theorem
\ref{lem:HH-main}, $J$ is a superhero and contains both $H$ and $T_k$
as subtournaments.  Let $D$ be a digraph with vertex set $V$
satisfying the hypotheses of the lemma.  A copy of $J$ in $D$ is
called a \emph{ball}. A ball $\hj$ is colored \emph{red} if
$\chi(M^+(\hj))\le d$ and \emph{blue} if $\chi(M^-(\hj))\le d$ (note
that we do not color vertices of a ball but color the ball as a single
object). A ball certainly can be both red and blue, in which case we
color it arbitrarily with one of the colors.

\begin{claim}\label{cl:uhk1-ball}
Every ball is either red or blue.
\end{claim}
\begin{cproof}
Suppose for a contradiction that a ball $\hj$ is neither red nor
blue. Then there are $B_1\subseteq M^-(\hj)$ and $B_2\subseteq
M^+(\hj)$ such that $\chi(B_1)=d$ and $\chi(B_2)=d$.  Suppose that
$\chi(N^+(v)\cap B_{1}) > c_1$ for some $v\in B_2$. Let $\hh$ be a
copy of $H$ in $N^+(v)\cap B_{1}$. Let $\hk$ be a copy of $T_k$ in
$\hj$. Then $\Delta(\hh,\hk,v)$ forms a copy of $\Delta(H,T_k,T_1)$, a
contradiction. Hence $\chi(N^+(v)\cap B_{1})\le c_1$ for every $v\in
B_2$. Similarly, $\chi(N^-(v)\cap B_{2})\le c_1$ for every $v\in B_1$,
for otherwise a copy of $T_k$ in $N^{-}(v) \cap B_2$ would yield a
contradiction. Hence $B_1,B_2$ is a $d$-bag-chain of length 2, a
contradiction.
\end{cproof}

For every vertex $v$ of $V$, we color $v$ as follows. If there are
$c_1+1$ vertex-disjoint red balls $\hj_1,\dots,\hj_{c_1+1}$ such that
$\hj_i$ has complete arcs to $\hj_j$ for every $i<j$, and $v\in
\hj_{c_1+1}$, then we color $v$ \emph{red}. If there are $c_1+1$ blue
balls $\hj_1,\dots,\hj_{c_1+1}$ such that $\hj_i$ has complete arcs to
$\hj_j$ for every $i<j$, and $v\in \hj_{1}$, then we color $v$
\emph{blue}.  If $v$ satisfies both conditions, we color $v$
arbitrarily.  After the process of coloring, we obtain a partition of
$V$ into $R$ the set of red vertices, $B$ the set of blue vertices,
and $U$ the set of uncolored vertices.

\begin{claim}
There is $d_1$ such that $\chi(U)\le d_1$.
\end{claim}
\begin{cproof}
Let $K$ be the tournament of $J\Rightarrow J\Rightarrow \dots
\Rightarrow J$ ($2c_1+2$ times $J$). Since $J$ is a superhero, then so
is $K$ by Theorem~\ref{lem:HH-main}. Hence there is $d_1$ such that
every $K$-free digraph $D'$ with $\alpha(D')\le \alpha$ has chromatic
number at most $d_1$.  Suppose that $U$ contains a copy of $K$, say
$\hat{K}$. Since every ball is either red or blue, we can find $c_1+1$
vertex-disjoint monochromatic balls $\hj_1,\dots,\hj_{c_1}$ in $\hat
K$ such that $\hj_i$ has complete arcs to $\hj_j$ for every
$i<j$. Then either vertices of $\hj_1$ are blue or vertices of
$\hj_{c_1+1}$ are red, a contradiction with the fact that all vertices
of $U$ are uncolored. Hence $U$ is $K$-free, and so $\chi(U)\le d_1$.
\end{cproof}

It remains to show that $R$ and $B$ have bounded chromatic number,
which can be done by applying Theorems \ref{thm:global}. To do so, we
need to prove that $N^+(v)$ has bounded chromatic number for every
$v\in R$.
\begin{claim}\label{cl:redball}
There is $d_2$ such that $\chi(N^+(v))\le d_2$ for every $v\in R$.
\end{claim}
\begin{cproof}

Fix $v\in R$. Then there are vertex-disjoint red balls
$\hj_1,\dots,\hj_{c_1+1}$ where $v\in \hj_{c_1+1}$ and $\hj_i$ has
complete arcs to $\hj_j$ for every $i< j$. Let
$L=\bigcup_{i=1}^{c_1}\hj_1$. Note that $v$ is seen by all vertices of
$L$.  For every $u\in L$, we have $\chi( N^+(v)\cap N^o(u))\le c_0$.
Hence
$$ \chi(N^+(v)\cap N^o(L))\le \sum_{u\in L} \chi(N^+(v)\cap N^o(u))\le |L| c_0.$$

For every partition of $L$ into $L_1,L_2$ ($L_1$ or $L_2$ may be
empty), let $Y_{L_1,L_2}=N^+(v)\cap M^-(L_1)\cap M^+(L_2)$. Observe
that for every vertex $x\in N^+(v)\setminus{N^o(L)}$ (note that
$x\notin L$ since $v\in \hj_{c_1+1}$), there is a partition of $L$
into some $L_1,L_2$ such that $x$ sees all vertices of $L_1$ and is
seen by all vertices of $L_2$, and so $x\in Y_{L_1,L_2}$. Hence
$$N^+(v)\setminus{N^o(L)}=\bigcup_{(L_1,L_2)}Y_{L_1,L_2}.
$$

We now show that $\chi(Y_{L_1,L_2})\le d$ for every $Y_{L_1,L_2}$.

\begin{itemize}
\item If, for some $1 \leq i \leq c_1$, there is $\hj_i\subseteq L_2$,
  recall that $\chi(M^+(\hj_i))\le d$ since $\hj_i$ is red.  Since
  $Y_{L_1,L_2}\subseteq M^+(L_2)\subseteq M^+(\hj_i)$, we have that
  $\chi(Y_{L_1,L_2})\le d$.

\item Otherwise, if there is no $\hj_i\subseteq L_2$, then $L_1$
  contains at least one vertex of each $\hj_i$, $1 \leq i \leq c_1$,
  and so $|L_1|\ge c_1$. Hence $L_1$ has a copy of $T_k$, say
  $\hk$. Note that all vertices of $\hk$ see $v$ since $v\in
  \hj_{c_1+1}$. If $\chi(Y_{L_1,L_2})> c_1$, then $Y_{L_1,L_2}$
  contains a copy of $H$, say $\hh$, then $\Delta(\hh,\hk,v)$ forms a
  copy of $\Delta(H,T_k,T_1)$, a contradiction. Hence
  $\chi(Y_{L_1,L_2})\le c_1\le d$.
\end{itemize}

Note that $|L|= |J|c_1=(h+k)c_1$. Thus, there are $2^{(h+k)c_1}$ possible ways to partition  $L$ into $L_1,L_2$. 
Hence 
\begin{align*}
\chi(N^+(v)) \le&\  \chi\Big(N^+(v)\cap N^o(L)\Big)+\chi\Big(N^+(v)\setminus{N^o(L)}\Big)\\ 
\le&\  |L|c_0+ \sum_{(L_1,L_2)}\chi(Y_{L_1,L_2})\\
\le&\ {(h+k)c_1} c_0 +2^{(h+k)c_1}d.
\end{align*}
Set $d_2:={(h+k)c_1}c_0+ 2^{{(h+k)c_1}}d.$ This completes the proof of the claim.
\end{cproof}
Then $\chi(N^+(v)\cap R)\le \chi(N^+(v))\le d_2$ for every $v\in R$. Then by applying Theorem~\ref{thm:local-global-main} for digraph $R$ with $t=d_2$, we have $\chi(R)\le d_3$ for some $d_3$. 

We now prove that $B$ has bounded chromatic number. The proof is
slightly different from that of Claim \ref{cl:redball} due to
asymmetry of $\Delta(H,T_k,T_1)$.

\begin{claim}
There is $d_4$ such that $\chi(N^-(v))\le d_4$ for every $v\in B$.
\end{claim}
\begin{cproof}
Fix $v\in B$. Then there are vertex-disjoint blues balls
$\hj_1,\dots,\hj_{c_1+1}$ where $v\in \hj_{1}$ and $\hj_i$ has
complete arcs to $\hj_j$ for every $i< j$. Let
$L=\bigcup_{i=2}^{c_1+1}\hj_i$. Note that $v$ sees all vertices of
$L$.  For every $u\in L$, we have
$$ \chi(N^-(v)\cap N^o(L))\le \sum_{u\in L} \chi(N^-(v)\cap N^o(u))\le |L| c_0.$$

For every partition of $L$ into $L_1,L_2$ ($L_1$ or $L_2$ may be
empty), let $Y_{L_1,L_2}=N^-(v)\cap M^-(L_1)\cap M^+(L_2)$. Observe
that for every vertex $x\in N^-(v)\setminus{N^o(L)}$ (note that
$x\notin L$ since $v\in \hj_{1}$), there is a partition of $L$ into
some $L_1,L_2$ such that $x$ sees all vertices of $L_1$ and is seen by
all vertices of $L_2$, and so $x\in Y_{L_1,L_2}$. Hence
$$N^-(v)\setminus{N^o(L)} = \bigcup_{(L_1,L_2)}Y_{L_1,L_2}.
$$

We now show that $\chi(Y_{L_1,L_2})\le d$ for every $Y_{L_1,L_2}$.

\begin{itemize}

\item If there is $\hj_i\subseteq L_2$, then $\hj_i$ contains a copy
  of $H$, say $\hh$. Note that $v$ is in the first ball, so $v$ sees
  all vertices of $\hh$. If $\chi(Y_{L_1,L_2}) > c_1$, then
  $Y_{L_1,L_2}$ contains a copy of $T_k$, say $\hk$, then
  $\Delta(\hh,\hk,v)$ forms a copy of $\Delta(H,T_k,T_1)$, a
  contradiction. Hence $\chi(Y_{L_1,L_2})\le c_1\le d$.

\item If $\hj_{c_1+1}\subseteq L_1$, recall that $\chi(M^-(\hj_{c_1+1}))\le d$ since $\hj_{c_1+1}$ is blue.
Since $Y_{L_1,L_2}\subseteq M^-(L_1)\subseteq M^-(\hj_{c_1+1})$, we have that $\chi(Y_{L_1,L_2})\le d$.

\item Otherwise, we have two remarks: (1) $\hj_{c_1+1}$ must have a
  vertex in $L_2$, say $z$. (2) Every $\hj_{i}, 2\le i\le c_1$ must
  have a vertex in $L_1$. Then $|L_1\cup\{v\}| > c_1$, and so
  $L_1\cup\{v\}$ contains a copy of $T_k$, say $\hk'$, such that all
  vertices of $\hk'$ are in one of the $\hj_i$, $1 \leq i \leq c_1$
  (note that $v\in \hj_1$). Observe that $z$ is seen by all vertices
  of $\hk'$ since $z\in \hj_{c_1+1}$. If $\chi(Y_{L_1,L_2}) > c_1$,
  then $Y_{L_1,L_2}$ contains a copy of $H$, say $\hh'$, then
  $\Delta(\hh',\hk',z)$ forms a copy of $\Delta(H,T_k,T_1)$, a
  contradiction. Hence $\chi(Y_{L_1,L_2})\le c_1\le d$.
\end{itemize}

Hence using a computation similar to the claim above, we have $\chi(N^-(v))\le d_4$, where $d_4:= {(h+k)c_1}c_0+ 2^{{(h+k)c_1}}d.$ This completes the proof of the claim.
\end{cproof}
Then $\chi(N^-(v)\cap B)\le \chi(N^-(v))\le d_4$ for every $v\in
B$. Then by applying Theorem~\ref{thm:local-global-main-in} to the
digraph $B$ with $t=d_4$, we have $\chi(B)\le d_5$ for some $d_5$.
Hence $$\chi(D)\le \chi(B)+\chi(R)+\chi(U)\le d_1+d_3+d_5.$$ This
completes the proof of Lemma~\ref{lem:h1k-2chain}.
\end{proof}

We are now ready to prove Lemma \ref{zone:bounded}.

\begin{claim}\label{cl:h1k-3zone}
If a $\Delta(H,T_k,T_1)$-free digraph $D$ 
with $\alpha(D)\le \alpha$ contains no $c$-bag-chain of length $6$, then 
$\chi(D)\le g(g(g(c)))$ where $g$ is the function in Lemma~\ref{lem:h1k-2chain}. 
\end{claim}
\begin{cproof}
Suppose for a contradiction that $\chi(D)\ge g(g(g(c)))$. We will show that $D$ contains a  $c$-bag-chain of length 8. By applying Lemma~\ref{lem:h1k-2chain} to $D$ with $d:=g(g(c))$, we have that $D$ contains a $g(g(c))$-bag-chain of length two, say $X_1,X_2$. Hence $\chi(X_1)=\chi(X_2)= g(g(c))$ and 
\begin{itemize}
\item $\chi(N^+(v)\cap X_1)\le c_1$ for every $v\in X_2$, and
\item $\chi(N^-(v)\cap X_2)\le c_1$ for every $v\in X_1$.
\end{itemize}

Apply Lemma~\ref{lem:h1k-2chain} again to $X_1$ (respectively, $X_2$)
with $d:=g(c)$, we obtain a $g(c)$-bag-chain of length two, say
$Y_1,Y_2$ in $X_1$ (respectively, $Y_3,Y_4$ in $X_2$).  Since
$Y_2\subset X_1$ and $Y_3\subset X_2$, we have
\begin{itemize}
\item $\chi(N^+(v)\cap Y_2)\le c_1$ for every $v\in Y_3$, and
\item $\chi(N^-(v)\cap Y_3)\le c_1$ for every $v\in Y_2$.
\end{itemize}
Hence by definition, $Y_1,\dots,Y_4$ forms a $g(c)$-bag-chain of length 4. 
Note that $\chi(Y_s)=g(c)$ for every $1\le s\le 4$.
Repeating the argument we obtain a $c$-bag-chain of length 2 inside each 
$Y_s$, and hence obtain a $c$-bag-chain $B_1,\dots,B_8$ of length 8 inside $D$. This contradicts the fact
that $D$ has no $c$-bag-chain of length 6, and so completes the proof. 
\end{cproof}

Claim~\ref{cl:h1k-3zone} proves Lemma \ref{zone:bounded}, concluding the proof
of Theorem~\ref{thm:h1k-main2}.



\section{Triangle-free digraphs} \label{sec:triangle-free}

In this section, we prove Theorem \ref{thm:c3}.  We present an
efficient algorithm to color a $C_3$-free digraph whose independence
number is $\alpha(D) \leq \alpha$ with at most $35^{\alpha-1}\alpha!$
colors.  For a digraph $D$, let $n = |V(D)|$ denote the size of its
vertex set.  Let $poly(n)$ denote the function $n^k$ for some rational
number $k > 0$.  With respect to the time complexity of our algorithm,
our main goal is to show that it is a polynomial depending only on $n$
(i.e., $poly(n)$).  We therefore do not optimize the running time nor
do we provide a tight analysis.  Moreover, note that each time we
argue that a subroutine can be performed in $poly(n)$ time, the
(implicit) value of $k$ may be different.

For each integer $\alpha\ge 1$, define $h(\alpha)$ to be the minimum
number such that every $C_3$-free digraph $D$ with $\alpha(D)\le
\alpha$ has chromatic number at most $h(\alpha)$. Conjecture
\ref{conj:poly-C_3} asserts that $h(\alpha)\le \alpha^{\ell}$ for some
${\ell}$. Clearly $h(1)=1$ since every $C_3$-free tournament is
acyclic. However, $h(\alpha)$ is still unknown for all $\alpha \ge
2$. We believe that $h(2)=2$ or $3$ even though the best bound we have
is around 25 (by tweaking the proof of Theorem
\ref{thm:c3}). Theorem \ref{thm:c3} gives an exponential upper bound
for $h(\alpha)$ by the function $g(\alpha) := 35^{\alpha-1}\alpha!$.
Since our proof of Theorem \ref{thm:c3} will use induction, we will
assume that a $C_3$-free digraph with independence number $\alpha-1$
can be colored with at most $g(\alpha-1)$ colors and we will prove it
for $\alpha$.  Note that $g(1) = 1$.  Let us restate Theorem \ref{thm:c3}.

\begin{theorem}\label{thm:c3-res}
Let $D$ be a $C_3$-free digraph where $\alpha(D)\le \alpha$.  Then
$\chi(D) \leq 35^{\alpha-1}\alpha!$ and this coloring can be found in
time $poly(n)$.
\end{theorem}

The rest of this section is devoted to proving Theorem
\ref{thm:c3-res}.  We begin with some observations regarding the size
of a dominating set in a digraph.

\begin{proposition}\label{prop:dom-to-acyclic}
A digraph $D$ has an acyclic dominating set, and this set can be
found in time $poly(n)$.
\end{proposition}
\begin{proof}
We proceed by induction on $n$. The statement clearly holds for
$n=1$. For $n>1$, pick an arbitrary vertex $v$. Then
$V(D)\setminus{v}=N^o(v)\cup N^-(v)\cup N^+(v)$. Applying induction to
the subgraph, $D[N^o(v)\cup N^-(v)]$, we obtain an acyclic dominating
set $S'$. Then $S:=S'\cup \{v\}$ is a dominating set of $D$.  Note
that $S'\subseteq N^o(v)\cup N^-(v)$, so $v$ does not see any vertex
of $S'$. Hence $S$ is an acyclic set since $S'$ is an acyclic set.
The running time for this procedure is $poly(n)$.
\end{proof}

\begin{proposition}\label{prop:qua-dom}
Given a $C_3$-free digraph $D$, there is a set $Y\subseteq V(D)$ of
size at most $\alpha(D)$ such that $V(D)=Y\cup N^o(Y)\cup N^+(Y)$, and
this set can be found in time $poly(n)$.
\end{proposition}

\begin{proof}
First apply Proposition \ref{prop:dom-to-acyclic} to obtain an acyclic
dominating set $S$ of $D$. Then apply Proposition
\ref{prop:dom-acyclic} to $D[S]$ to obtain a stable dominating set $Y$
of $D[S]$ of size at most $\alpha (D)$ in time $poly(n)$.

We now show that $Y$ is a set with the desired properties.  Suppose
for a contradiction that there is $v\notin Y\cup N^o(Y)\cup
N^+(Y)$. Then $Y\subseteq N^+(v)$ and $v\notin S$ since $Y$ dominates
all vertices of $S$. There is $u\in S$ seeing $v$ since $S$ is a
dominating set of $D$. Note that $u\notin Y$; otherwise this
contradicts $Y\subseteq N^+(v)$. There is $y\in Y$ seeing $u$ since
$Y$ is a dominating set of $S$. Then $u,v,y$ are distinct vertices
where $u$ sees $v$, $v$ sees $y$, and $y$ sees $u$. Hence we obtain a
copy of $C_3$ in $D$, a contradiction.
\end{proof}

We now present some definitions and useful lemmas.  First, we
re-define a {\em bag} so that it retains the useful properties of a
bag as defined in Section \ref{sec:H1K} and so that it can be tested
efficiently.

\begin{definition}
For a digraph $D$, we say that $B \subseteq V(D)$ is a {\em bag} of $D$, if
every three distinct vertices $\{x,y,z\} \in V(D)\setminus{B}$ 
have a common neighbor in $B$.
\end{definition}
Recall that $u$ and $v$ are neighbors if either $uv$ or $vu$ is an
arc.  We can check in $poly(n)$ time (e.g. $O(n^4)$) whether or not a
set $B$ is a bag of $D$ by exhaustively checking all triples in $V(D)
\setminus B$.  Suppose that the $n$ vertices of a digraph can be
partitioned into disjoint sets such that the subgraph induced on each
set has independence number at most $\alpha-1$.  Let $t(\alpha-1,n)$
denote the maximum (over all such possible partitions of the vertices)
total time required by our algorithm (the algorithm {\sc
  Color-Digraph}, which we will define shortly) to color all of the
subgraphs, each with at most $g(\alpha-1)$ colors.  The following claim
follows from this definition.

\begin{claim}\label{cl:runtime}
Suppose $\sum_{i=1}^{\ell} n_i = n$, where $n_i \geq 1$ is an integer.
Then $\sum_{i=1}^{\ell} t(\alpha,n_i) \leq t(\alpha,n)$.
\end{claim}

We now fix an arbitrary $C_3$-free digraph $D$ such that $\alpha(D)\le
\alpha$, and we omit the subscript $D$ from the relevant notation when
the context is clear.

\begin{claim}\label{cl:notbag}
If $S \subset V(D)$ is not a bag of $D$, then:

\begin{enumerate}[label=(\alph*)]
\item \label{notbag1} We can partition $S$ into three disjoint sets,
  $S_1, S_2$ and $S_3$, such that $\alpha(D[S_i]) \leq \alpha-1$ for
  $i \in \{1,2,3\}$.  This procedure takes time $poly(n)$.

\item \label{notbag2} We can color $S$ with $3 \cdot g(\alpha-1)$
  colors in time $t(\alpha-1,|S|)$.
\end{enumerate}
\end{claim}

\begin{cproof}
If $S$ is not a bag of $D$, then we can, in time $poly(n)$, find a
triple $\{x,y,z\} \in V(D) \setminus S$ such that every $v\in S$ is
not incident to at least one of $x,y$ or $z$.  Thus, each vertex $v
\in S$ belongs to either $N^o(x)\cap S, N^o(y)\cap S$ or $N^o(z)\cap
S$.  Each of these sets has independence number at most $\alpha-1$.
The total time for this procedure is $poly(n)$.  The second assertion
follows from \ref{notbag1} and from the definition of the function
$t(\alpha,n)$.
\end{cproof}

\begin{definition}
A bag $B \subseteq V(D)$ is {\em poor} if for every vertex $v \in B$,
either $N^-(v)\cap B$ or $N^+(v)\cap B$ is not a bag of $D$.
\end{definition}
We can check in $poly(n)$ time (e.g. $O(n^5)$) if a bag $B$ is poor by
testing whether or not $N^-(v)\cap B$ and $N^+(v)\cap B$ are bags for
every $v\in B$.

\begin{claim}\label{cl:rho}
If $B \subseteq V(D)$ is a poor bag, then we can color $B$ with
$8 \alpha \cdot g(\alpha-1)$ colors in time $poly(n) + t(\alpha-1, |B|)$.
\end{claim}

\begin{cproof}
Since $B$ is poor, then for each $v \in B$, either $N^-(v)\cap B$ or
$N^+(v)\cap B$ is not a bag of $D$.  Then we can partition $B$ into
two sets, $L$ and $R$, where $N^-(v)\cap B$ is not a bag for every $v\in L$,
and $N^+(v)\cap B$ is not a bag for every $v\in R$.

Applying Proposition \ref{prop:qua-dom} to $R$, we can find a set $Y_R
\subseteq R$ (respectively $Y_L \subseteq L$) such that $|Y_R| \leq \alpha$ and $R \subseteq Y_R \cup
N^+(Y_R) \cup N^o(Y_R)$.  And so:
$$R=\bigcup_{v\in Y_R}\Big(\big(N^+(v)\cup N^o(v)\cup \{v\}\big)\cap
R\Big).$$ For $v \in Y_R$, note that $N^+(v)\cap R$ is not a bag, and so by
\ref{notbag1} from Claim \ref{cl:notbag}, we can partition $N^+(v)
\cap R$ into three sets, each with independence number at most
$\alpha-1$.  Additionally, we have the set $N^o(v) \cap R$, which also
has independence number at most $\alpha-1$.
Overall, we can partition $R\setminus{Y_R}$ into $4
\alpha$ sets, each with independence number $\alpha-1$.  The same
argument can be applied to $Y_L$ to obtain $8 \alpha$ disjoint sets.
Therefore, in time $poly(n)$, we can partition $B \setminus{Y_R \cup
  Y_L}$, and in time $t(\alpha-1, |B|)$ we can color $B\setminus{\{Y_r
  \cup Y_L\}}$ with $8 \alpha \cdot g(\alpha-1)$ colors.  We can then
color each $v \in Y_R$ (respectively, $v \in Y_L$) with an arbitrary
color used to color the set $N^o(v) \cap R$ (respectively, $N^o(v)
\cap L$).
\end{cproof}

In general, we do not know how to color a bag efficiently, and a bag
may be very large (e.g. $V(D)$ is a bag).  Our aim is therefore to find poor
bags, since these can be colored using Claim \ref{cl:rho}.  The
first step of our algorithm is to find a chain of poor bags.

\begin{definition}
A sequence of pairwise disjoint bags $B_1, \dots ,B_t$ forms a \emph{chain
  of bags} if $B_i\to
B_{i+1}$ for every $i \in [1,t).$ \end{definition} 

Recall that $B_i\to B_{i+1}$ means there is no arc from $B_{i+1}$ to
$B_{i}$.  Moreover, if each $B_i$ is a poor bag, then this sequence
is a {\em chain of poor bags}.  Given a chain of bags $C = \{B_1,B_2,
\dots, B_t\}$ for $D$, we say that $v \in C$ if $v \in B_i$ for some
$i \in [1,t]$.  We can partition the vertices in $V(D) \setminus{C}$
into sets $Z_0, \dots ,Z_t$, which we call \emph{zones}, as follows.
For every $v\in V(D) \setminus{C}$, let $i$ be the largest index such
that $v$ is seen by at least one vertex in $B_i$.  Then vertex $v$ is
assigned to zone $Z_i$.  Otherwise, we assign $v$ to zone $Z_0$. This
partition is unique and can be done in time $poly(n)$.  As in the case
of the bags and zones used in Section \ref{sec:H1K}, these bags and
zones exhibit useful properties.  The proofs we present here are
similar, but much simpler.

\begin{claim}\label{cl:c3-wit}
Let $C = \{B_1, \dots, B_t\}$ be a chain of bags, and let $Z_0, Z_1, \dots,
Z_t$ be a partition of the vertices in $V(D) \setminus{C}$.
For every $i$, the following properties hold:
\begin{enumerate}[label=(\alph*)]
\item \label{en:c3.1} $B_i\to B_{i+r}$ for every $r\ge 1$, 
\item \label{en:c3.2} $Z_i\to B_{i+r}$ for every $r\ge 1$, 
\item \label{en:c3.3} $B_i\to Z_{i+r}$ for every $r\ge 2$,
\item \label{en:c3.4} $Z_i\to Z_{i+r}$ for every $r\ge 3$.
\end{enumerate}
\end{claim}

\begin{cproof}
Property \ref{en:c3.1} holds for $r = 1$ by definition of a chain of
bags.  Suppose that \ref{en:c3.1} holds for $r-1 > 1$, and suppose
that there is an arc $uv$ with $u\in B_{i+r}$ and $v\in B_i$. Since
$B_{i+1}$ is a bag, there is $x\in B_{i+1}$ such that $x$ is a common
neighbor of $u$ and $v$. Then by induction hypothesis, $vx,xu$ are
arcs, and so $vxu$ is a copy of $C_3$, a contradiction. Hence
\ref{en:c3.1} holds for $r$.

Property \ref{en:c3.2} holds for all $r\ge 1$ by the partitioning
criterion of vertices into zones.

To prove property \ref{en:c3.3}, suppose that there is an arc $zv$
with $z\in Z_{i+r}$ and $v\in B_i$ for some $r\ge 2$.  Then there is
$u\in B_{i+r}$ such that $uz$ is an arc by the partitioning criterion
of vertices into zones.  Since $B_{i+1}$ is a bag, there is $x\in
B_{i+1}$ such that $x$ is a common neighbor of $u,v,z$.  By property
\ref{en:c3.1}, $vx$ and $xu$ are arcs. If $xz$ is an arc, then $vxz$
is a copy of $C_3$. Otherwise, $zx$ is an arc, and so $xuz$ is a copy
of $C_3$. Either way, we reach the contradiction, and so \ref{en:c3.3}
holds for every $r\ge 2$.

To prove property \ref{en:c3.4}, suppose that there is an arc $uv$
with $u\in Z_{i+r}$ and $v\in Z_i$ for some $r\ge 3$. Since $B_{i+1}$
is a bag, there is $x\in B_{i+1}$ such that $x$ is a common neighbor
of both $u$ and $v$.  By property \ref{en:c3.2}, $vx$ is an arc, and by
property \ref{en:c3.3}, $xu$ is an arc. Hence $vxu$ is a copy of
$C_3$, a contradiction. Hence \ref{en:c3.4} holds for every $r\ge 3$.
\end{cproof}

We now show how to find a chain of poor bags, which can be colored
using Claim \ref{cl:rho}.

\vspace{3mm}

\begin{mdframed}

\vspace{2mm}
{\sc Find-Chain}$(D,B)$
\begin{enumerate}
\item {If there is $v \in B$ such that both $N^+(v)\cap B$ and
  $N^-(v)\cap B$ are bags of $D$, then:

~~~~ Return \big({\sc Find-Chain}$(D,N^-(v)\cap B)$,~$v$, {\sc Find-Chain}$(D,N^+(v)\cap B)$ \big).

\item Otherwise, return $B$.
}
\end{enumerate}

\end{mdframed}

\vspace{3mm}

The routine {\sc Find-Chain}$(D,B)$ returns
$B_1,v_1,B_2,\dots,v_{t-1},B_t$ (i.e., a sequence of sets $B_i$
alternating with vertices $v_i$).  We say that the sequence $B_1, B_2,
\dots, B_t$ is the chain of poor bags output by the procedure {\sc
  Find-Chain}$(D,B)$.  Later on, we will use the vertices $v_i$ in the
output sequence to facilitate the coloring of vertices outside the
chain.  Observe that if $B$ is a poor bag or if $B$ is not a bag, then
{\sc Find-Chain}$(D,B)$ returns a single set, namely $B$.

\begin{claim} \label{cl:findchain}
If bag $B \subseteq V(D)$ is not poor, then {\sc Find-Chain}$(D,B)$
returns a chain of poor bags $B_1, \dots ,B_t$ for some $t \ge 2$ in time
$poly(n)$.
\end{claim}

\begin{cproof}
Let $B_1, \dots, B_t$ be the chain of poor bags output by {\sc
  Find-Chain}$(D,B)$.  The bags in this chain are pairwise disjoint.
From Step 1 of {\sc{Find-Chain}}, it follows that each $B_i$ is a
bag. Furthermore, each $B_i$ must be poor; otherwise, Step 1 would be
applied to $B_i$ to return poor bags inside $B_i$.  Observe that for
every pair of consecutive bags $B_i, B_{i+1}$ in the chain,
$B_{i+1}\subseteq N^+(v_i)$ and $B_i\subseteq N^-(v_i)$. If there is
an arc $xy$ with $x\in B_{i+1}$ and $y\in B_{i}$, then $v_ixy$ is a
copy of $C_3$, a contradiction. Hence $B_1,\dots,B_t$ is a chain of
poor bags by definition.  The procedure {\sc Find-Chain}$(D,B)$ runs
in time $poly(n)$.
\end{cproof}

\begin{claim}\label{cl:c3-zone}
If {\sc{Find-Chain}}$(D, B)$ returns a chain of $t$ poor bags, then we
can color $B$ with $8\alpha \cdot g(\alpha-1) + (t-1)\cdot
g(\alpha-1)$ colors in time $poly(n) + t(\alpha-1, |B|)$.
\end{claim}

\begin{cproof}
Suppose that {\sc{Find-Chain}}$(D,B)$ returns $B_1,v_1,B_2,\dots,
v_{t-1},B_t$ and that $C= \{B_1, \dots, B_t\}$ is the chain of poor
bags output by the procedure.  Since each $B_i$ is a poor bag, it can be colored using
Claim \ref{cl:rho}.  By Claim \ref{cl:c3-wit}, $B_i\to B_j$ for every
$i<j$, and so we can color the vertices in $C$ using $8 \alpha \cdot
g(\alpha-1)$ colors in time $poly(n) + t(\alpha-1, |C|)$ time.

Observe that each $v \in B \setminus{C}$ is either (i) some $v_i$ in
the output sequence returned by {\sc Find-Chain}$(D,B)$, or (ii)
belongs to $N^o(v_i)$ for some $v_i$ in this output sequence.  For any
$v \in V(D)$, the set $N^o(v)$ has independence number $\alpha-1$.
Therefore, the vertices in $\bigcup_{i=1}^{t-1} N^o(v_i)$ can be colored
with $(t-1) \cdot g(\alpha-1)$ colors.  Note that $v_i$ can be colored
with an arbitrary color from the color palette used for $N^o(v_i)$.
This coloring can be found in time $poly(n) + t(\alpha-1, |B
\setminus{C}|)$.
\end{cproof}

\begin{corollary}\label{cor:chain-col}
Either {\sc Find-Chain}$(D,B)$ returns a chain of $t$ poor bags, or
$B$ can be colored using $8 \alpha \cdot g(\alpha-1) + (t-2) \cdot
g(\alpha-1)$ colors.
\end{corollary}

We now have the tools to outline our main coloring algorithm.

\vspace{2mm}
\noindent
\fbox{\parbox{15cm}{
{\sc Color-Digraph}$(D)$

\vspace{2mm}
If $D$ is acyclic, color $D$ with one color and terminate.

\vspace{2mm}
Otherwise: 
\begin{enumerate}

\item Run {\sc Find-Chain}$(D,V(D))$ and 
let $C := \{B_1, \dots, B_t\}$ denote the chain of poor bags that is output.

\item Assign each vertex in $V(D) \setminus{C}$ to a zone $Z_i$ for $i
  \in [0,t]$.  

\item {While {\sc Find-Chain}$(D,Z_i)$ returns a chain of poor bags $B_1', B_2', \dots, B_k'$ for $k \geq 3$:

\begin{enumerate}

\item Update chain: $C := \{B_1, \dots, B_{i-2}, B_1', B_2', \dots,
  B_k', B_{i+1}, \dots, B_t\}.$

\item Re-assign each vertex in $V(D) \setminus{C}$ to a zone.

\end{enumerate}
}

\item Color all vertices in the chain $C$ with $8\alpha \cdot g(\alpha-1)$ colors.

\item Color all vertices in the zones of $C$ with $3(8 \alpha + 1)\cdot g(\alpha-1)$ colors.

\end{enumerate}
}}

\vspace{3mm}

\begin{claim}\label{cl:main}
{\sc Color-Digraph}$(D)$ colors $D$ with at most $35\alpha \cdot
g(\alpha-1)$ colors.
\end{claim}

\begin{cproof}
If $V(D)$ is a poor bag, then we can apply Claim \ref{cl:rho}.  Hence
we may suppose that $V(D)$ is not a poor bag.  

Note that the updated chain $C$ resulting from Step 3 (a) is still a
chain of poor bags, due to properties \ref{en:c3.2} and \ref{en:c3.3}
from Claim \ref{cl:c3-wit}.  After Step 3 finishes, the chain $C$ is {\em
  maximal} in that the procedure {\sc Find-Chain}$(D,Z_i)$ will not
find a chain of three poor bags in any zone $Z_i$.  Using Corollary
\ref{cor:chain-col}, we can therefore color each zone using $(8 \alpha
+ 1)\cdot g(\alpha -1)$ colors.  Applying property \ref{en:c3.4} from
Claim \ref{cl:c3-wit}, we can use at most $(24 \alpha + 3) \cdot
g(\alpha-1)$ colors to color all vertices in $V(D)\setminus{C}$.  Each
bag $B_i$ in the chain $C$ is a poor bag, so we can color $B_i$ with
$8\alpha\cdot g(\alpha-1)$ colors by Claim \ref{cl:rho}.  Moreover,
since $B_i\to B_j$ for every $i<j$ (by property \ref{en:c3.1} from
Claim \ref{cl:c3-wit}), we need $8\alpha\cdot g(\alpha-1)$ colors to
color the entire chain $C$.  Thus, we can color $D$ with
$(32\alpha+3)\cdot g(\alpha-1) \leq 35 \alpha \cdot g(\alpha-1)$
colors.
\end{cproof}

\begin{claim}\label{cl:g-bound}
The procedure {\sc Color-Digraph}$(D)$ uses $g(\alpha)$ colors.
\end{claim}

\begin{cproof}
We proceed by induction on $\alpha$.  If $\alpha=1$, then we use one
color, since every $C_3$-free tournament is acyclic.  Suppose that the
algorithms colors each $C_3$-free digraph $D'$ where $\alpha(D')\le
\alpha-1$ using at most $g(\alpha-1)$ colors.  Then by Claim
\ref{cl:main}, the procedure {\sc Color-Digraph}$(D)$ colors $D$ with at most
\begin{eqnarray*}
35 \alpha \cdot g(\alpha-1) & = & 35 \alpha \cdot
  35^{\alpha-2}(\alpha - 1)! ~ = ~ 35^{\alpha-1} \alpha! 
\end{eqnarray*}
colors.
\end{cproof}

\begin{claim}\label{cl:run-time-alpha}
{\sc Color-Digraph}$(D)$ runs in time $poly(n)$.
\end{claim}

\begin{cproof}
We now analyze the running time.  Finding a maximal chain $C$ takes
$poly(n)$ time, as does the procedure of partitioning the vertices in
$V(D) \setminus{C}$ into zones.  Once we have found this partitioning,
we can color the vertices in $C$ in time $poly(n) + t(\alpha-1, |C|)$
using Claim \ref{cl:rho}, and we can color each zone $Z_i$ in time
$poly(n) + (\alpha-1, |Z_i|)$ using Claim \ref{cl:c3-zone} and
Corollary \ref{cor:chain-col}.  So applying Claim \ref{cl:runtime},
the total running time of {\sc Color-Digraph}$(D)$ is at most $poly(n) + t(\alpha-1,n)$.
This leads to the following recurrence relation:
\begin{eqnarray*}
t(\alpha,n) & = & poly(n) + t(\alpha-1,n)\\
& = & \alpha \cdot poly(n).
\end{eqnarray*}
Since $\alpha \leq n$, we have that $t(\alpha,n) = poly(n)$.

We note that our algorithm is actually just partitioning $V(D)$ into
disjoint subsets, where each subset has independence number at most
$\alpha-1$.  In other words, it first partitions the vertices in
$V(D)$ into sets where each set is a poor bag or not a bag (for
example, a zone $Z_i$ is either a poor bag or not a bag).  Then, it
further partitions these sets into sets with independence number at
most most $\alpha-1$.  (For example, by Claim \ref{cl:notbag}, a set
that is not a bag can be partitioned into three disjoint sets, each
with independence number at most $\alpha-1$.  Similarly, in the proof
of Claim \ref{cl:rho}, a poor bag is partitioned into $8 \alpha$
disjoint sets, each with independence number at most $\alpha-1$.)
Once the subgraphs on these induced subsets are colored (recursively)
using $g(\alpha-1)$ colors, then certain subsets are allowed to use
the same color palette, and these color palettes can be coordinated in
time $poly(n)$.  
The initial partitioning procedure and the final coordinating procedure require
$poly(n)$ time, while the recursive coloring requires $t(\alpha-1,n)$ time.  
\end{cproof}

Theorem \ref{thm:c3-res} follows from Claims \ref{cl:main},
\ref{cl:g-bound} and \ref{cl:run-time-alpha}.

Finally, we remark that Theorem \ref{thm:c3} yields a bound on the
size of a maximum acyclic subgraph of a $C_3$-free digraph $D$ in
terms of $\alpha$.

\begin{theorem}
Let $D=(V,A)$ be a $C_3$-free digraph where $\alpha(D) \leq \alpha$.
Then $D$ contains an acyclic subset of arcs, $A' \subseteq A$, with
cardinality $|A'| \geq |A| \cdot (\frac{1}{2} + c_{\alpha})$, where
$c_{\alpha}$ is a constant depending on $\alpha$.
\end{theorem}

\begin{proof}
Since $D$ can be colored with $g(\alpha)$ colors, there is a color
class containing at least $n/g(\alpha)$ vertices, and by Turan's
Theorem, its induced subgraph contains at least
\begin{eqnarray*}
\frac{n}{2 \cdot g(\alpha)} \left(\frac{n}{\alpha \cdot g(\alpha)} - 1\right)
\end{eqnarray*}
arcs.  Thus, $D$ contains a maximum acyclic subgraph of size at least
\begin{eqnarray*}
\frac{|A|}{2} + 
\frac{n}{4 \cdot g(\alpha)}\left(\frac{n}{\alpha \cdot
  g(\alpha)}-1 \right).
\end{eqnarray*}
Let $c$ be an absolute constant.  Then
\begin{eqnarray*}
c_{\alpha} = \frac{1}{c \cdot \alpha \cdot g(\alpha)^2}
\end{eqnarray*}
satisfies the theorem.
\end{proof}

\end{document}